\documentclass[smallextended,numbook,runningheads]{svjour3}
\usepackage{graphicx,amsmath,amsfonts}
\usepackage{mathptmx}
\usepackage{ifpdf}
\usepackage{color}

\usepackage[ruled,vlined,lined,linesnumbered,commentsnumbered]{algorithm2e}
\DeclareMathAlphabet{\mathcal}{OMS}{cmsy}{m}{n}

\usepackage{nicefrac} 
\usepackage{url}

\def\FmuL{\mathcal{F}_{\mu,L}}
\def\FzL{\mathcal{F}_{0,L}}
\def\half{{\textstyle\frac{1}{2}}}
\def\Qset{\mathcal{Q}}

\def\PQ{P_{\Qset}}
\def\TV{T_{\tau}}
\def\TTVV{\mathcal{T}}
\def\TTVVtau{\mathcal{T}_{\tau}}
\def\vec{\mathrm{vec}}
\def\rank{\mathop{\rm rank}}

\def\algname{{\bf UPN}}
\def\algnameR{{\bf R}\algname}
\def\algnameBB{{\bf GPBB}}
\def\algnameGP{{\bf GP}}
\def\algnameNesterov{{\bf Nesterov}}
\def\algnameBT{{\bf BT}}
\def\algnameUPNz{{\bf UPN}$_0$}
\def\fig{Fig. }

\def\rmax{M}		
\def\vidx{j}		
\def\vmax{N}		
\def\D1{D_1}		
\def\Dend{D_N}		
\def\Drun{D_\vidx}	

\def\Tover{{\bf T1}}
\def\Tunder{{\bf T2}}

\newcommand{\ie}{{ i.e.}}

\newcommand{\argmin}{\mathop{\rm argmin}}

\newcommand{\Xfun}{\mathcal{X}}
\newcommand{\OC}{\mathcal{O}}

\definecolor{red}{rgb}{1,0,0}
\definecolor{orange}{rgb}{1,0.4,0}
\definecolor{pink}{rgb}{1,0,0.7}

\definecolor{green}{rgb}{0,0.5,0}
\definecolor{cyan}{rgb}{0.125,0.7,0.667}
\definecolor{purple}{rgb}{0.5,0,0.5}

\allowdisplaybreaks

\begin{document}

\title{Implementation of an Optimal First-Order Method for Strongly Convex
  Total Variation Regularization
\thanks{This work is part of the project
  CSI:\ Computational Science in Imaging, supported by grant no.\ 274-07-0065
  from the Danish Research Council for Technology and Production Sciences.}
}

\author{T. L. Jensen  \and
        J. H. J{\o}rgensen    \and
        P. C. Hansen \and S. H. Jensen
}

\authorrunning{T. L. Jensen, J. H. J{\o}rgensen, P. C. Hansen, and
  S. H. Jensen}

\titlerunning{A First-Order Method for Strongly Convex TV Regularization}

\institute{Tobias Lindstr{\o}m Jensen, S{\o}ren Holdt Jensen \at
           Department of Electronic Systems, Aalborg University, Niels Jernesvej 12,
           9220 Aalborg \O, Denmark. \\
           \email{\{tlj,shj\}@es.aau.dk}
  \and
           Jakob Heide J{\o}rgensen, Per Christian Hansen \at
           Department of Informatics and Mathematical Modelling,
           Technical University of Denmark, Building 321,
           2800 Lyngby, Denmark. \\
           \email{\{jakj,pch\}@imm.dtu.dk}
}

\date{Received: date / Accepted: date}

\maketitle

\begin{abstract}
We present a practical implementation of an optimal first-order method, due to Nesterov, for large-scale total variation
regularization in tomographic reconstruction, image deblurring, etc. The algorithm applies to $\mu$-strongly convex objective functions with $L$-Lipschitz continuous gradient. In the framework of Nesterov both $\mu$ and $L$ are assumed known -- an assumption that is seldom satisfied in practice. We propose 
to incorporate mechanisms to estimate locally sufficient $\mu$ and $L$ during the iterations. The mechanisms also allow for the application to non-strongly convex functions. We discuss the iteration
complexity of several first-order methods, inclu\-ding the proposed algorithm, and
we use a 3D tomography problem to compare the performance of these methods.
The results show that for ill-conditioned problems
solved to high accuracy, the proposed method significantly outperforms
state-of-the-art first-order methods, as also suggested by theoretical results.
\end{abstract}


\section{Introduction} \label{intro}
Large-scale discretizations of inverse problems \cite{Hansen} arise in
a variety of applications such as medical imaging,
non-destructive testing, and geoscience.
Due to the inherent instability of these problems, it is necessary
to apply regularization in order to compute meaningful
reconstructions, and this work focuses on the use of total variation
which is a powerful technique when the sought solution is required
to have sharp edges (see, e.g., \cite{Chan:05,Rudin92}
for applications in image reconstruction).


Many total variation algorithms have already been developed, such as
time marching \cite{Rudin92}, fixed-point iteration \cite{Vogel:96},
and various minimization-based methods such as sub-gradient methods
\cite{Alter:05,Combettes:02}, second-order cone programming (SOCP)
\cite{Goldfarb:05}, duality-based methods
\cite{Chambolle:04,Chan:98,Hintermuller:06}, and graph-cut
methods \cite{Chambolle:05,Darbon:06}.

The difficulty of a problem depends on the linear operator to be inverted. Most methods are dedicated to denoising, where the operator is simply the identity, or possibly deblurring of a simple blur, where the operator is invertible and represented by a fast transform. For general linear operators with no exploitable structure, such as in tomographic reconstruction, the selection of algorithms is limited. Furthermore, the systems that arise in real-world tomography applications, especially in 3D, are so large that memory-requirements preclude the use of second-order methods with quadratic convergence.
%
%

Recently, Nesterov's optimal first-order method \cite{Ne:04,Ne:05}
has been adapted to, and analyzed for, a number of imaging problems
\cite{Dahl:10,Weiss:09}.
In \cite{Weiss:09} it is shown that Nesterov's method
outperforms standard first-order methods by an order of magnitude,
but this analysis does not cover tomography problems.
A drawback of Nesterov's algorithm (see, e.g., \cite{Chambolle:2010})
is the explicit need for
the strong convexity parameter and the Lipschitz constant of the objective function, both of which are not available in practice.


This paper describes a practical implementation of Nesterov's algorithm,
augmented with efficient heuristic methods to estimate the unknown Lipschitz constant and strong convexity parameter.
The Lipschitz constant is handled using backtracking, similar to
the technique used in~\cite{beck:2009}.
To estimate the unknown strong convexity parameter -- which is more
difficult -- we propose a heuristic based on adjusting an estimate of
the strong convexity parameter using a local strong convexity inequality. Furthermore, we equip the heuristic with a restart procedure to ensure convergence in case of an inadequate estimate.

We call the algorithm \algname{} (Unknown Parameter Nesterov) and compare it
with two versions of the well-known gradient projection
algorithm; \algnameGP{}: a simple version using a backtracking line search for the stepsize and \algnameBB{}: a more advanced version using Barzilai-Borwein acceleration \cite{BB:88} and with
the backtracking procedure from \cite{Grippo:86}.
We also compare with a variant of the proposed algorithm, \algnameUPNz{}, where the strong convexity information is not enforced. This variant is similar to the FISTA algorithm \cite{beck:2009}. We have implemented the four algorithms in C with a MEX interface to MATLAB, and the software is available from \url{www2.imm.dtu.dk/~pch/TVReg/}.

Our numerical tests demonstrate that the proposed method \algname{} is significantly faster than \algnameGP{}, and as fast as \algnameBB{}
for moderately ill-conditioned problems, and
significantly faster for ill-con\-di\-tioned problems. Compared to \algnameUPNz, \algname{} is con\-sist\-ently faster, when solving to high accuracy.

We start with introductions to the discrete total variation problem,
to smooth and strongly convex functions, and to some basic
first-order methods in Sections \ref{sec:TVproblem}, \ref{sec:optimal},
and \ref{sec:basic}, respectively.
Section \ref{sec:inequalities} introduces important inequalities
while the new algorithm is described in Section~\ref{sec:UPN}.
The 3D tomography test problem is introduced in Section~\ref{sec:tomo}.
Finally, in Section \ref{sec:numerical} we report our numerical tests
and comparisons.

Throughout the paper we use the following notation.
The smallest singular value of a matrix $A$ is denoted $\sigma_{\min}(A)$. 
The smallest and largest eigenvalues of a sym\-metric semi-definite matrix $M$
are denoted by~$\lambda_{\min}(M)$ and $\lambda_{\max}(M)$.
For an optimization pro\-blem, $f$ is the objective function,
$x^\star$ denotes a minimizer, $f^\star = f(x^\star)$ is the optimum
objective, and $x$ is called an \hbox{$\epsilon$-suboptimal} solution if $f(x)-f^\star \leq \epsilon$. 

\section{The Discrete Total Variation Reconstruction Problem} \label{sec:TVproblem}

The Total Variation (TV) of a real function $\Xfun(t)$ with
$t\in\mathrm{\Omega}\subset\mathbb{R}^p$ is defined as
  \begin{align}
    \TTVV(\Xfun) = \int_{\mathrm{\Omega}} \| \nabla \Xfun (t) \|_2 \, dt.
  \end{align}
Note that the Euclidean norm is not squared, which means that $\TTVV(\Xfun)$
is non-differen\-tiable.
In order to handle this we consider a smoothed version of the TV functional.
Two common choices are to replace the Euclidean norm of the vector $z$ by either
$(\| z \|_2^2 + \beta^2)^{1/2}$ or the Huber function
  \begin{equation} \label{eq:huberdef}
    \Phi_\tau (z) = \left\{ 
\begin{array}{lll}
      \| z \|_2 - \half\tau & \quad \text{if} & \| z \|_2 \geq \tau, \\[2mm]
      \frac{1}{2\tau} \| z \|_2^2 & \quad \text{else}.&
   \end{array} 
\right.
\end{equation}
In this work we use the latter, which can be considered a prox-function
smoothing \cite{Ne:05} of the TV functional \cite{becker:2009}; thus,
the approximated TV functional is given by
  \begin{equation}
  \label{eq:tvtaudef}
    \TTVVtau(\Xfun) = \int_{\mathrm{\Omega}} \Phi_\tau \left(\nabla \Xfun \right) \, dt.
  \end{equation}

In this work we consider the case $t\in\mathbb{R}^3$.
To obtain a discrete version of the TV reconstruction problem, we represent
$\Xfun(t)$ by an $\vmax = m \times n \times l$ array $X$, and we let $x = \vec(X)$.
Each element or voxel of the array $X$, with index $\vidx$, has an associated
matrix (a discrete differential operator) $\Drun \in \mathbb{R}^{3 \times \vmax}$
such that the vector $\Drun\, x\in\mathbb{R}^3$ is the forward difference
approximation to the gradient at $x_\vidx$.
By stacking all $\Drun$ we obtain the matrix $D$ of dimensions $3\vmax \times \vmax$:
  \begin{equation}\label{eq:Dform}
    D = \left(\begin{array}{c} \D1 \\ \vdots \\ \Dend \end{array} \right).
  \end{equation}
We use periodic boundary conditions in $D$, which ensures that only a
constant $x$ has a TV of $0$. Other choices of boundary conditions could easily
be implemented.

When the discrete approximation to the gradient is used and the integration in
\eqref{eq:tvtaudef} is replaced by summations, the discrete and smoothed TV
function is given by
  \begin{equation}
  \label{eq:smoothedTV}
    \TV(x) = \sum_{\vidx=1}^{\vmax} \Phi_\tau(\Drun x) .
  \end{equation}
The gradient $\nabla T_\tau(x)\in\mathbb{R}^N$ of this function is given by
  \begin{equation}
    \nabla T_\tau(x) = \sum_{j=1}^N D_j^T D_j\, x / \max \{ \tau, \| D_j x \|_2 \} .
  \end{equation}

We assume that the sought reconstruction has voxel values
in the range $[0,1]$, so we wish to solve a bound-constrained problem, i.e.,  having the feasible region 
\hbox{$\Qset = \{ x \in \mathbb{R}^\vmax ~\vert~ 0 \leq x_\vidx \leq 1 \}$.}
Given a linear system $A\, x \approx b$ where $A \in \mathbb{R}^{\rmax\times \vmax}$
and \hbox{$\vmax=mnl$},
we define the associated \textit{discrete TV regularization problem} as
  \begin{equation}
  \label{eq:discreteTV}
    x^\star = \argmin_{x\in \Qset} \ \phi(x), \qquad
    \phi(x) = \half \| A\, x-b \|_2^2 + \alpha\, \TV(x) ,
  \end{equation}
where $\alpha > 0$ is the TV regularization parameter.
This is the problem we want to solve, for the case where the linear
system of equations arises from discretization of an inverse problem.

\section{Smooth and Strongly Convex Functions} \label{sec:optimal}

To set the stage for the algorithm development in this paper,
we consider the convex optimization problem $\min_{x\in \Qset} f(x)$
where $f$ is a convex function and $\Qset$ is a convex set.
We recall that a continuously differentiable function $f$ is convex if
  \begin{align} 
    f(x) \geq f(y) + \nabla f(y)^T (x-y), \qquad \forall \, x,y \in \mathbb{R}^\vmax .
  \end{align} 
\begin{definition}
A continuously differentiable convex function $f$ is said to be
\emph{strongly convex} with \emph{strong convexity parameter} $\mu$ if there exists a  $\mu \geq 0$ such that
  \begin{equation}
  \label{eq:defstronglyconvex}
    f(x) \geq f(y) + \nabla f(y)^T (x-y) + \half\mu \| x-y\|_2^2 ,
    \qquad \forall \, x,y \in \mathbb{R}^\vmax .
  \end{equation}
\end{definition}
\begin{definition}
A continuously differentiable convex function $f$ has \emph{Lipschitz continuous}
gradient with \emph{Lipschitz constant} $L$, if
  \begin{equation}
  \label{eq:defsmooth}
    f(x) \leq f(y) + \nabla f(y)^T(x-y) + \half L \|x-y\|_2^2 ,
    \qquad \forall \, x,y \in \mathbb{R}^\vmax .
  \end{equation}
\end{definition}
\begin{remark}
The condition \eqref{eq:defsmooth} is equivalent \cite[Theorem 2.1.5]{Ne:04}
to the more standard way of defining Lipschitz continuity of the gradient,
namely, through convexity and the condition
$\| \nabla f(x) - \nabla f(y) \|_2 \leq L \|x-y\|_2 , \forall \, x,y
\in \mathbb{R}^\vmax$.
\end{remark}
\begin{remark}
Lipschitz continuity of the gradient is a smoothness requirement on~$f$.
A function $f$ that satisfies \eqref{eq:defsmooth} is said to be smooth, and
$L$ is also known as the \emph{smoothness constant}.
\end{remark}

The set of functions that satisfy \eqref{eq:defstronglyconvex}
and (\ref{eq:defsmooth}) is denoted $\FmuL$.
It is clear that $\mu \leq L$ and also that if $\mu_1 \geq \mu_0$ and
$L_1 \leq L_0$ then $f\in \mathcal{F}_{\mu_1,L_1} \Rightarrow f\in
\mathcal{F}_{\mu_0,L_0}$.
Given fixed choices of $\mu$ and $L$, we introduce the ratio $Q = L/\mu$
(sometimes referred to as the ``modulus of strong convexity'' \cite{Nemirovsky:1983} or the
``condition number for~$f$'' \cite{Ne:04}) which is an
upper bound for the condition number of the Hessian matrix.
The number $Q$ plays a major role for the convergence rate of optimization
methods we will consider.
%
\begin{lemma}
\label{lem:quadraticmuL}
For the quadratic function $f(x) = \half \| A\, x - b \|_2^2$ with
$A\in\mathbb{R}^{\rmax\times \vmax}$ we have
  \begin{equation}
L = \| A \|_2^2, \qquad 
    \mu = \lambda_{\min}(A^TA) = 
\left\{ 
\begin{array}{lll}
      \sigma_{\min}(A)^2  &  \quad \mathrm{if} & \mathrm{rank}(A) = \vmax, \\
      0 , & \quad \mathrm{else,}  &
\end{array} 
\right.
  \end{equation}
and if $\mathrm{rank}(A)=N$ then $Q=\kappa(A)^2$, the square of
the condition number of~$A$.
\end{lemma}
\begin{proof}
Follows from 
$f(x) = f(y) + (A\, y - b)^T A (x-y) + \half (x-y)^T A^T A (x-y)$,
the \hbox{second} order Taylor expansion of $f$ about $y$, where equality holds
for quadratic $f$.  \hfill $\Box$
\end{proof}
\begin{lemma}
\label{lem:TVmuL}
For the smoothed TV function \eqref{eq:smoothedTV} we have
  \begin{equation}
    L = \| D \|_2^2 / \tau , \qquad \mu = 0 ,
  \end{equation}
where $\| D \|_2^2 \leq 12$ in the 3D case.
\end{lemma}
\begin{proof}
The result for $L$ follows from \cite[Thm.~1]{Ne:05} since the
smoothed TV functional can be written as \cite{becker:2009,Dahl:10}
  \[
    \TV(x) = \max_u \left \{ \, u^TDx - \frac{\tau}{2} \, \| u \|_2^2\, : \,
    \| u_i \|_2 \leq 1, \ \forall\, i=1,\ldots, N  \, \right \}
\]
with $u=(u_1^T,\ldots,u_N^T)^T$ stacked according to $D$.
The inequality $\|D\|_2^2 \leq 12$ follows from a straightforward
extension of the proof in the Appendix of \cite{Dahl:10}. For $\mu$ pick  \hbox{$y = \alpha e\in \mathbb{R}^N$} and $x=\beta e\in \mathbb{R}^N$, where
$e=(1,\ldots,1)^T$, and $\alpha\neq\beta \in \mathbb{R}$.
Then we get \hbox{$\TV(x)=\TV(y)=0$}, $\nabla \TV(y)=0$ and obtain
  \[
    \half\mu \|x-y\|_2^2 \leq \TV(x) - \TV(y) - \nabla \TV(y)^T(x-y) = 0,
  \]
and hence $\mu=0$.
\hfill $\Box$
\end{proof}
\begin{theorem}
\label{thm:Qtight}
For the function $\phi(x)$ defined in \eqref{eq:discreteTV} we have a strong
convexity parameter $\mu = \lambda_{\min}(A^TA)$ and Lipschitz constant
$L = \| A \|_2^2 + \alpha\, \| D \|_2^2 / \tau$.
If $\mathrm{rank}(A)<N$  then $\mu =0$, otherwise $\mu=\sigma_{\min}(A)^2>0$
and
  \begin{equation}
  \label{eq:Qf}
    Q = \kappa(A)^2 + \frac{\alpha}{\tau} \,
    \frac{\| D \|_2^2}{\sigma_{\min}(A)^2} ,
  \end{equation}
where $\kappa(A) = \|A\|_2 / \sigma_{\min}(A)$ is the condition number of $A$.
\end{theorem}
\begin{proof}
Assume $\rank(A) = \vmax$ and consider $f(x) = g(x) + h(x)$ with
$g\in\mathcal{F}_{\mu_g,L_g}$ and $h\in\mathcal{F}_{\mu_h,L_h}$.
Then $f \in \mathcal{F}_{\mu_f,L_f}$, where $\mu_f = \mu_g + \mu_h$
and $L_f = L_g + L_h$.
From $\mu_f$ and $L_f$ and using Lemmas \ref{lem:quadraticmuL} and
\ref{lem:TVmuL} with $g(x) = \half \| A\, x-b \|_2^2$ and
$h(x) = \alpha \TV(x)$ we obtain the condition number for $\phi$
given in \eqref{eq:Qf}.
If $\rank(A)<N$ then the matrix $A^TA$ has at least one zero eigenvalue,
and thus $\mu=0$.
\hfill $\Box$
\end{proof}
\begin{remark} \label{remark:nontight}
Due to the inequalities used to derive \eqref{eq:Qf}, there is no
guarantee that the given $\mu$ and $L$ are the tightest possible for~$\phi$.
For $\rank(A)<N$ there exist problems for which the Hessian matrix is singular
and hence $\mu=0$, but we cannot say if this is always the case.
\end{remark}

\section{Some Basic First-Order Methods} \label{sec:basic}
A basic first-order method is the gradient projection method of the form
  \begin{align} \label{eq:gradientmethod}
    x^{(k+1)} = \PQ \left(x^{(k)} - \theta_k \nabla f(x^{(k)}) \right ) ,
    \qquad k = 0,1,2,\ldots \ .
  \end{align}
The following theorem summarizes the convergence properties.
\begin{theorem} \label{thm:gradientmethod}
Let $f \in \FmuL $, $\theta_k = 1/L$ and $x^\star\in \Qset$ be the
constrained minimizer of $f$, then for the gradient method
\eqref{eq:gradientmethod} we have
  \begin{equation}
    f(x^{(k)}) - f^\star \leq \frac{L}{2 k} \|x^{(0)} - x^\star\|_2^2 .
  \label{eq:gradientrate2}
  \end{equation}
Moreover, if $\mu\neq 0$ then
  \begin{equation}
    f(x^{(k)}) - f^\star \leq  \left(1-\frac{\mu}{L} \right)^k
    \bigl( f(x^{(0)}) - f^\star \bigr) .
  \label{eq:gradientrate1}
  \end{equation}
\end{theorem}
\begin{proof}
The two bounds follow from \cite{Va:09} and \cite[\S 7.1.4]{Nemirovsky:1983},
respectively. \hfill $\Box$
\end{proof}

To improve the convergence of the gradient (projection) method, Barzilai
and Borwein \cite{BB:88} suggested a scheme in which the step
$\theta_k \nabla f(x^{(k)})$ provides a simple and computationally
cheap approximation to the Newton step $(\nabla^2 f(x^{(k)}))^{-1} \nabla
f(x^{(k)})$.  For general unconstrained problems with $f\in\FmuL$,
possibly with $\mu=0$, non-monotone line search combined with the Barzilai-Borwein (BB)
strategy produces algorithms that converge \cite{Raydan:97}; but it is difficult to
give a precise iteration complexity for such algorithms.  For strictly
quadratic unconstrained problems the
BB strategy requires $\OC \left( Q \log \epsilon^{-1} \right)$
iterations to obtain an \hbox{$\epsilon$-suboptimal} solution \cite{Dai:2002}.  In
\cite{Fletcher:90} it was argued that, in practice, $\OC\left( Q \log
  \epsilon^{-1} \right)$ iterations ``is the best that could be
expected''.  This comment is also supported by the statement in
\cite[p.~69]{Ne:04} that all ``reasonable step-size rules'' have the
same iteration complexity as the standard gradient method.  Note that
the classic gradient method \eqref{eq:gradientmethod} has $\OC
(L/\epsilon)$ complexity for $f\in \FzL$.  To summarize, when using
the BB strategy we should not expect better complexity than $\OC
(L/\epsilon)$ for $f\in \FzL$, and $\OC \left( Q \log \epsilon^{-1}
\right)$ for $f\in \FmuL$.

In Algorithm \ref{algo:BB} we give the (conceptual) algorithm \algnameBB{}, which implements the BB strategy
with non-monotone line search \cite{Birgin:00,Zhu:08} using the
backtracking procedure from \cite{Grippo:86} (initially combined in~\cite{Raydan:97}).
The algorithm needs the real parameter \hbox{$\sigma \in [0,1]$} and
the nonnegative integer~$K$, the latter specifies the number of iterations over which
an objective decrease is guaranteed.
\begin{algorithm}[tb]\normalsize 
\SetKwInOut{Input}{input}\SetKwInOut{Output}{output}\SetKw{Abort}{abort}
\Input{$x^{(0)}$, $K$}
\Output{$x^{(k+1)}$}
$\theta_0 = 1$ \;
\For{$k = 0, 1, 2, \dots$}{
\tcp{BB strategy}
\If{$k>0$}{
$\theta_k \leftarrow \frac{\|x^{(k)} - x^{(k-1)} \|_2^2}{\langle \,
  x^{(k)} - x^{(k-1)}, \nabla f(x^{(k)}) - \nabla f(x^{(k-1)}) \, \rangle}$ \;
  } 
  $\beta\leftarrow 0.95$ \;
  $\bar{x} \leftarrow \PQ( x^{(k)} - \beta \theta_k \nabla f(x^{(k)}) )$ \;
  $\hat{f} \leftarrow \max\{ f(x^{(k)}), f(x^{(k-1)}),\ldots,f(x^{(k-K)}) \}$ \;
  \While{$f(\bar{x}) \geq \hat{f} - \sigma \,
  \nabla f(x^{(k)})^T(x^{(k)}-\bar{x})$}{
  $\beta \leftarrow \beta^2$ \;
  $\bar{x} \leftarrow \PQ(x^{(k)} - \beta\theta_k \nabla f(x^{(k)}))$ \;
  } 
  $x^{(k+1)} \leftarrow \bar{x}$ \; } 
\caption{\algnameBB{}}\label{algo:BB}
\end{algorithm}

An alternative approach is to consider first-order methods with optimal
complexity. The optimal complexity is defined as the worst-case
complexity for a first-order method applied to any problem in a
certain class \cite{Nemirovsky:1983,Ne:04} (there are also more
technical aspects involving the problem dimensions and a black-box
assumption). In this paper we focus on the classes $\FzL$ and~$\FmuL$.

Recently there has been a great deal of interest
in optimal first-order methods for convex optimization problems with
$f\in\FzL$ \cite{beck:2009b,Tse:08}.
For this class it is possible to reach an $\epsilon$-suboptimal so\-lu\-tion
within $\OC(\sqrt{L/\epsilon})$ iterations.
Nesterov's methods can be used as stand-alone optimization algorithm,
or in a composite objective
setup \cite{beck:2009,Ne:07,Tse:08}, in which case
they are called accelerated methods (because the designer violates the
black-box assumption).
Another option is to apply optimal first-order methods to a smooth approximation
of a non-smooth function leading to an algorithm with
$\OC \left ( 1/\epsilon\right)$ complexity \cite{Ne:05}; for practical considerations, see \hbox{\cite{becker:2009,Dahl:10}}.

%
%
Optimal methods specific for the function class $\FmuL$ with $\mu>0$
are also known \cite{Ne:83,Ne:04}; see also \cite{Ne:07} for the composite
objective version. However, these methods have gained little practical consideration; for example in \cite{Ne:07} all the simulations are conducted with $\mu=0$.
Optimal methods require $\OC \left( \sqrt{Q} \log \epsilon^{-1} \right)$
iterations while the classic gradient method requires
$\OC \left( Q \log \epsilon^{-1} \right)$ iterations \cite{Nemirovsky:1983,Ne:04}.
For quadratic problems, the conjugate gradient method achieves the same iteration
complexity as the optimal first-order method~\cite{Nemirovsky:1983}.

In Algorithm \ref{algo:Nesterov} we state the basic optimal method \algnameNesterov{} \cite{Ne:04} with known $\mu$ and $L$;
it requires an initial $\theta_0\geq \sqrt{\mu/L}$.
Note that it uses two sequences of vectors, $x^{(k)}$ and $y^{(k)}$.
\begin{algorithm}[tb]\normalsize
\SetKwInOut{Input}{input}\SetKwInOut{Output}{output}\SetAlgoSkip{smallskip}
\Input{$x^{(0)}$, $\mu$, $L$, $\theta_0$}
\Output{$x^{(k+1)}$}
 $y^{(0)} \leftarrow x^{(0)}$\;
\For{$k = 0, 1, 2, \dots$}{
$x^{(k+1)} \leftarrow \PQ \bigl(y^{(k)} - L^{-1} \nabla f(y^{(k)}) \bigr)$ \label{algoline:nesterovxk1}\;
$\theta_{k+1} \leftarrow \mathrm{positive}~\mathrm{root}~\mathrm{of}~
  \theta^2 = (1-\theta)\theta_k^2 + \frac{\mu}{L}\theta$ \;
$\beta_k \leftarrow \theta_k(1-\theta_k)/(\theta_k^2 + \theta_{k+1})$ \;
$y^{(k+1)} \leftarrow x^{(k+1)} + \beta_k (x^{(k+1)} - x^{(k)})$ \;
} 
\caption{\algnameNesterov{}}\label{algo:Nesterov} 
\end{algorithm}
The convergence rate is provided by the following theorem.
\begin{theorem} \label{thm:nesterov}
If $f \in \FmuL $, $1>\theta_0 \geq \sqrt{\mu/L}$, and
$\gamma_0 = \frac{\theta_0(\theta_0L-\mu)}{1-\theta_0}$, then for
algorithm \algnameNesterov{} we have
  \begin{equation} \label{eq:nesterovrate}
    f(x^{(k)}) - f^\star \leq \frac{4L}{(2\sqrt{L} + k\sqrt{\gamma_0})^2}
    \left( f(x^{(0)}) - f^\star +
    \frac{\gamma_0}{2}\|x^{(0)} - x^\star\|_2^2\right).
  \end{equation}
Moreover, if $\mu\neq 0$
  \begin{equation} \label{eq:nesterovrate_2}
    f(x^{(k)}) - f^\star \leq \left(1-\sqrt{\frac{\mu}{L}}\right)^k
    \left( f(x^{(0)}) - f^\star +
    \frac{\gamma_0}{2}\|x^{(0)} - x^\star\|_2^2\right).
  \end{equation}
\end{theorem}
\begin{proof}
See \cite[(2.2.19), Theorem 2.2.3]{Ne:04} and Appendix \ref{app:optimal_convergence_rate}
for an alternative proof.
\hfill $\Box$
\end{proof}

Except for different constants Theorem \ref{thm:nesterov} mimics the result in
Theorem \ref{thm:gradientmethod}, with the crucial differences that the denominator
in \eqref{eq:nesterovrate} is squared and $\mu/L$ in \eqref{eq:nesterovrate_2}
has a square root.
Comparing the convergence rates in Theorems \ref{thm:gradientmethod}
and \ref{thm:nesterov}, we see that the rates are linear but differ in the linear
rate, $Q^{-1}$ and $\sqrt{Q^{-1}}$, respectively.
For ill-conditioned problems, it is important whether the complexity is
 a function of $Q$ or $\sqrt{Q}$, see, e.g., \cite[\S 7.2.8]{Nemirovsky:1983}.
This motivates the interest in specialized optimal first-order methods
for solving ill-conditioned problems.

\section{First-Order Inequalities for the Gradient Map} \label{sec:inequalities}

For unconstrained convex problems the (norm of) the gradient is a measure of how
close we are to the minimum, through the first-order optimality condition,
cf.~\cite{BoVa:04}.
For constrained convex problems $\min_{x\in\Qset} f(x)$ there is a similar
quantity, namely, the \emph{gradient map} defined by
  \begin{align}
    G_\nu(x) = \nu \left (x - \PQ \left (x-\nu^{-1}\nabla f(x) \right) \right) .
  \end{align}
Here $\nu>0$ is a parameter and $\nu^{-1}$ can be interpreted as the step size
of a gradient step.
The function $\PQ$ is the Euclidean projection onto the convex
set $\Qset$ \cite{Ne:04}.
The gradient map is a generalization of the gradient to constrained problems
in the sense that if $\Qset = \mathbb{R}^\vmax$ then $G_\nu(x)=\nabla f(x)$,
and the equality $G_\nu(x^\star) = 0$ is a necessary and sufficient optimality
condition~\cite{Va:09}.
In what follows we review and derive some important first-order
inequalities which will be used to analyze the proposed algorithm.
We start with a rather technical result.
\begin{lemma}\label{lemma_gradient_map}
Let $f\in \FmuL$, fix $x\in \Qset$, $y\in \mathbb{R}^\vmax$,
and set $x^+ = \PQ(y-\bar L^{-1}\nabla f(y))$, where $\bar\mu$ and
$\bar L$ are related to $x,y$ and $x^+$ by the inequalities
  \begin{align}
  \label{eq:barmu}
    f(x) & \geq f(y) + \nabla f(y)^T(x-y) + \half \bar{\mu} \|x-y \|_2^2, \\[2mm]
  \label{eq:barL}
    f(x^+) & \leq f(y) + \nabla f(y)^T(x^+-y) + \half \bar{L} \|x^+-y \|_2^2 .
  \end{align}
Then
  \begin{equation}
    f(x^+)\leq f(x) + G_{\bar L}(y)^T(y-x) - \half \bar{L}^{-1}
    \|G_{\bar{L}}(y) \|_2^2 - \half \bar{\mu} \|y-x\|_2^2 .
  \end{equation}
\end{lemma}
\begin{proof}
Follows directly from \cite[Theorem 2.2.7]{Ne:04}. \hfill $\Box$
\end{proof}

Note that if $f\in\FmuL$, then in Lemma \ref{lemma_gradient_map} we can always select
$\bar{\mu} = \mu$ and $\bar{L} = L$ to ensure that the inequalities
\eqref{eq:barmu} and \eqref{eq:barL} are satisfied. However, for
specific $x$, $y$ and $x^+$, there can exist $\bar \mu \geq \mu$ and $\bar L \leq L$
such that \eqref{eq:barmu} and \eqref{eq:barL} hold. We will use
these results to design an algorithm for unknown parameters $\mu$ and $L$.

The lemma can be used to obtain the following lemma.
The derivation of the bounds is inspired by similar results for
composite objective functions in~\cite{Ne:07}, and the second
result is similar to~\cite[Corollary 2.2.1]{Ne:04}.
\begin{lemma}\label{lemma_gradient_map_sup3}
Let $f\in \FmuL$, fix $y\in\mathbb{R}^\vmax$, and set
$x^+ = \PQ(y-\bar{L}^{-1} \nabla f(y))$.
Let $\bar{\mu}$ and $\bar{L}$ be selected in accordance with
\eqref{eq:barmu} and \eqref{eq:barL} respectively. Then
  \begin{equation}
  \label{eq:strong_convexity_gradient_map}
    \half \bar{\mu} \| y-x^\star \|_2 \leq \| G_{\bar L}(y) \|_2 .
  \end{equation}
If $y\in \Qset$ then
  \begin{equation}
  \label{eq:strong2}
    \half \bar{L}^{-1} \|G_{\bar L}(y) \|_2^2\leq f(y)-f(x^+)\leq f(y)-f^\star .
  \end{equation}
\end{lemma}
\begin{proof}
From Lemma \ref{lemma_gradient_map} with $x=x^\star$ we use $f(x^+)\geq f^\star$ and obtain
  \[
    \half \bar{\mu} \|y-x^\star\|_2^2 \leq G_{\bar L}(y)^T(y-x^\star)
    - \half \bar{L}^{-1} \|G_{\bar L}(y) \|_2^2 \leq
    \|G_{\bar L}(y) \|_2 \|y-x^\star\|_2 ,
  \]
and \eqref{eq:strong_convexity_gradient_map} follows;
Eq.~\eqref{eq:strong2} follows from Lemma \ref{lemma_gradient_map} using
$y=x$ and $f^\star\leq f(x^+)$. \hfill $\Box$
\end{proof}

As mentioned in the beginning of the section, the results of the corollary say that we can relate the norm of the gradient map at $y$
to the error $\| y-x^* \|_2$ as well as to $f(y)-f^*$.
This motivates the use of the gradient map in a stopping criterion:
  \begin{align} \label{eq:stoppingcriterion}
    \|G_{\bar L}(y)\|_2 \leq \bar \epsilon ,
  \end{align}
where $y$ is the current iterate, and $\bar L$ is linked to this
iterate using \eqref{eq:barL}.
The parameter $\bar \epsilon$ is a user-specified tolerance based on the
requested accuracy.
Lemma \ref{lemma_gradient_map_sup3} is also used in the following
section to develop a restart criterion to ensure convergence.

\section{Nesterov's Method With Parameter Estimation} \label{sec:UPN}
The parameters $\mu$ and $L$ are explicitly needed in \algnameNesterov{},
but it is much too expensive to compute them explicitly; hence we need
a scheme to estimate them during the iterations.
To this end, we introduce the estimates $\mu_k$ and $L_k$ of $\mu$ and $L$
in each iteration~$k$.
We discuss first how to choose $L_k$, then $\mu_k$, and finally we state the
complete algorithm \algname{} and its convergence properties.

To ensure convergence, the main inequalities (\ref{eq:osc_ine1}) and
(\ref{eq:osc_ine2}) must be satisfied.
Hence, according to Lemma \ref{lemma_gradient_map} we need to choose
$L_k$ such that
  \begin{align}
    f(x^{(k+1)})\leq f(y^{(k)}) + \nabla f(y^{(k)})^T(x^{(k+1)}-y^{(k)}) +
    \half L_k \|x^{(k+1)}-y^{(k)} \|_2^2 .
  \end{align}
This is easily accomplished using \textit{backtracking} on $L_{k}$ \cite{beck:2009}.
The scheme, \algnameBT{}, takes the form given in Algorithm \ref{algo:BT}, where $\rho_L>1$ is an adjustment parameter.
\begin{algorithm}[tb]\normalsize
\SetKwInOut{Input}{input}\SetKwInOut{Output}{output}
\Input{$y,\bar{L}$}
\Output{$x,\tilde L$}
$\tilde L \leftarrow \bar{L}$ \;
$x \leftarrow \PQ \left( y-\tilde L^{-1} \nabla f(y) \right)$ \;
\While{$f(x)> f(y)+\nabla f(y)^T(x-y) + \half \tilde L \|x-y\|_2^2$}{
	$\tilde  L \leftarrow \rho_L \tilde L$ \;
	$x \leftarrow \PQ \left( y-\tilde L^{-1} \nabla f(y) \right)$ \;
} 
\caption{\algnameBT}\label{algo:BT}
\end{algorithm}
If the loop is executed $n_{\mathrm{BT}}$ times, the dominant computational
cost of \algnameBT{} is \hbox{$n_{\mathrm{BT}} + 2$} function evaluations and 1
gradient evaluation.

According to (\ref{eq:osc_ine2}) 
with Lemma \ref{lemma_gradient_map} and (\ref{eqn:per_iteration_relation}), we need to select the estimate
$\mu_k$ such that $\mu_k\leq \mu_k^\star$, where $\mu_k^\star$ satisfies
\begin{align}
    f(x^\star) \geq f(y^{(k)}) + \nabla f(y^{(k)})^T(x^\star-y^{(k)})
    + \half \mu_k^\star \|x^\star-y^{(k)} \|_2^2 .
\end{align}
However, this is not possible because $x^\star$ is, of course, unknown.
To handle this problem, we propose a \textit{heuristic} where we
select $\mu_k$ such that
\begin{align}
    f(x^{(k)}) \geq f(y^{(k)}) + \nabla f(y^{(k)})^T(x^{(k)}-y^{(k)})
    + \half \mu_k \|x^{(k)}-y^{(k)} \|_2^2 .
\end{align}
This is indeed possible since $x^{(k)}$ and $y^{(k)}$ are known iterates.
Furthermore, we want the estimate $\mu_k$ to be decreasing in order to
approach a better estimate of~$\mu$.
This can be achieved by the choice
\begin{align}
  \label{eq:mu_k_estimate}
    \mu_k = \min \{\mu_{k-1},M(x^{(k)},y^{(k)}) \} ,
\end{align}
where we have defined the function
\begin{align}
    M(x,y) = \left \{
    \begin{array}{lll}
      \frac{f(x)-f(y)-\nabla f(y)^T(x-y)}{\frac{1}{2}\|x-y
      \|_2^2}& \quad \text{if} &  x\neq y, \\[2mm]
      \infty & \quad \text{else.} &\end{array} \right.
\end{align}
In words, the heuristic chooses the largest $\mu_k$ that satisfies
\eqref{eq:defstronglyconvex} for $x^{(k)}$ and $y^{(k)}$, as long as $\mu_k$
is not larger than $\mu_{k-1}$.
The heuristic is simple and computationally inexpensive and we have found
that it is effective for determining a useful estimate.
Unfortunately, convergence of \algnameNesterov{} equipped with this heuristic
is not guaranteed, since the estimate can be too large.
To ensure convergence we include a restart procedure \algnameR{} that detects
if $\mu_k$ is too large, inspired by the approach in \cite[\S 5.3]{Ne:07} for composite objectives. \algnameR{} is given in Algorithm \ref{algo:RUPN}.

To analyze the restart strategy, assume that $\mu_i$ for all $i=1,\ldots,k$
are \emph{small enough}, \ie, they satisfy $\mu_i\leq\mu_i^\star$ for
$i=1,\ldots,k$, and $\mu_k$ satisfies
  \begin{align}
    f(x^{\star}) \geq f(x^{(0)}) + \nabla f(x^{(0)})^T(x^\star-x^{(0)}) +
    \half \mu_k \|x^\star-x^{(0)} \|_2^2 .
  \end{align}
When this holds we have the convergence result (using
\eqref{eqn:convergence})
  \begin{align}\label{eq:convergence_shifted_one}
    f(x^{(k+1)})-f^\star \leq
    \prod_{i=1}^{k} \left (\!1-\!\sqrt{\mu_i/L_i} \right)
    \Bigl( f(x^{(1)})-f^\star + \half \gamma_1 \|x^{(1)}-x^\star\|_2^2 \Bigr).
  \end{align}
We start from iteration $k=1$ for reasons which will presented
shortly (see Appendix~\ref{app:optimal_convergence_rate} for details and definitions).
If the algorithm uses a projected gradient step from the initial $x^{(0)}$
to obtain $x^{(1)}$, the rightmost factor of (\ref{eq:convergence_shifted_one})
can be bounded as
  \begin{eqnarray}
  \nonumber
      \lefteqn{ f(x^{(1)})-f^\star + \half \gamma_1 \|x^{(1)}-x^\star\|_2^2 } \\
  \nonumber
      & \leq & G_{L_{0}}(x^{(0)})^T (x^{(0)}\! -\! x^\star)-\half L_{0}^{-1}
      \|G_{L_{0}}(x^{(0)})\|_2^2 + \half \gamma_1 \|x^{(1)}\!-\!x^\star \|_2^2 \\
  \nonumber
      & \leq & \| G_{L_{0}}(x^{(0)}) \|_2 \| x^{(0)}\! -\! x^\star \|_2 -
      \half L_{0}^{-1} \| G_{L_{0}}(x^{(0)}) \|_2^2
      + \half \gamma_1 \|x^{(0)}\!-\!x^\star \|_2^2 \\
  \label{eq:restart_bound_gradient_projection_reduces}
     & \leq & \left( \frac{2}{\mu_k} - \frac{1}{2L_0} + \frac{2\gamma_1}{\mu_k^2} \right )
    \|G_{L_{0}}(x^{(0)})\|_2^2 .
  \end{eqnarray}
Here we used Lemma \ref{lemma_gradient_map}, and the fact that a projected gradient step
reduces the Euclidean distance to the solution \cite[Theorem 2.2.8]{Ne:04}.
Using Lemma \ref{lemma_gradient_map_sup3} we arrive at the bound
  \begin{equation}
  \label{eq:iteration_mu_sufficiently_small}
    \half \tilde L_{k+1}^{-1} \| G_{\tilde L_{k+1}}(x^{(k+1)}) \|_2^2 \leq
    \prod_{i=1}^{k} \left (\!1-\!\sqrt{\frac{\mu_i}{L_i}}\right)
    \left( \frac{2}{\mu_k}-\frac{1}{2 L_{0}}+\frac{2 \gamma_1}{\mu_k^2} \right )
    \|G_{L_{0}}(x^{(0)})\|_2^2 .
  \end{equation}
If the algorithm detects that \eqref{eq:iteration_mu_sufficiently_small}
is not satisfied, it can only be because there was at least one $\mu_i$
for $i=1,\ldots,k$ which was \emph{not small enough}.
If this is the case, we restart the algorithm with a new
$\bar{\mu}\leftarrow \rho_\mu \mu_k$, where $0<\rho_\mu<1$ is a
parameter, using the current iterate $x^{(k+1)}$ as initial vector.

\begin{algorithm}[tb]\normalsize
\SetKw{Abort}{abort}\SetKw{Restart}{restart}
$ \gamma_1=\theta_1(\theta_1L_1-\mu_1)/(1-\theta_1) $\;
\If{$\mu_k\neq0 ~\mathbf{and}~$\rm inequality \eqref{eq:iteration_mu_sufficiently_small} not satisfied
    }{
    \Abort{} execution of \algname{}\;
    \Restart{} \algname{} \rm with input $(x^{(k+1)},\ \rho_\mu \mu_k, \
L_k, \ \bar{\epsilon})$\;
    } 
\caption{\algnameR}\label{algo:RUPN}
\end{algorithm}

The complete algorithm \algname{} (Unknown-Parameter Nesterov) is given in Algorithm \ref{algo:UPN}.
\algname{} is based on Nesterov's optimal method where
we have included backtracking on $L_k$ and the
heuristic~(\ref{eq:mu_k_estimate}). An initial vector $x^{(0)}$ and
initial parameters $\bar{\mu}\geq\mu$ and $\bar{L}\leq L$ must be
specified along with the requested accuracy $\bar \epsilon$.
\begin{algorithm}[tb]\normalsize
\SetKwInOut{Input}{input}\SetKwInOut{Output}{output}\SetKw{Abort}{abort}
\Input{$x^{(0)},\bar \mu,\bar L,\bar \epsilon$}
\Output{$x^{(k+1)}$ or $\tilde x^{(k+1)}$ }
$[ x^{(1)},L_{0} ] \leftarrow \algnameBT{}(x^{(0)},\bar{L})$ \label{algoline:upn_bt1}\;
$\mu_0=\bar \mu, \quad y^{(1)} \leftarrow x^{(1)} , \quad \theta_1
\leftarrow \sqrt{\mu_0/L_0}$ \;
\For{$k = 1, 2, \dots$}{
	$[ x^{(k+1)},L_k ] \leftarrow \algnameBT{}(y^{(k)},L_{k-1})$ \;
	$[\tilde x^{(k+1)},\tilde{L}_{k+1} ] \leftarrow \algnameBT{}(x^{(k+1)},L_k)$ \;
	\lIf{$\| G_{\tilde{L}_{k+1}}(x^{(k+1)}) \|_2 \leq
    \bar{\epsilon}$}{
    \Abort, \Return $\tilde x^{(k+1)}$ \;
    \lIf{$\| G_{L_k}(y^{(k)}) \|_2 \leq \bar{\epsilon}$}{
    \Abort, \Return $x^{(k+1)}$} \;
    $\mu_k \leftarrow \min \bigl\{ \mu_{k-1}, M(x^{(k)},y^{(k)}) \bigr\}$ \;
    \algnameR{}\;
    $\theta_{k+1} \leftarrow \mathrm{positive}~\mathrm{root}~\mathrm{of}~
    \theta^2 = (1-\theta)\theta_k^2 + (\mu_k/L_k)\, \theta$ \;
    $\beta_k \leftarrow \theta_k(1-\theta_k)/(\theta_k^2 + \theta_{k+1})$ \;
    $y^{(k+1)} \leftarrow x^{(k+1)} + \beta_k (x^{(k+1)} - x^{(k)})$ \;
    }
} 
\caption{\algname{}}\label{algo:UPN}
\end{algorithm}
The changes from \algnameNesterov{} to \algname{} are at the following lines:
\begin{description}
\item[\bf 1:] Initial projected gradient step to obtain the bound
  (\ref{eq:restart_bound_gradient_projection_reduces}) and thereby the
  bound (\ref{eq:iteration_mu_sufficiently_small}) used for the
  restart criterion.
\item[\bf 5:] Extra projected gradient step explicitly applied to obtain
  the stopping criterion $\| G_{\tilde{L}_{k+1}}(x^{(k+1)}) \|_2 \leq
  \bar \epsilon$.
\item[\bf 6,7:] Used to relate the stopping criterion in terms of $\bar
  \epsilon$ to $\epsilon$, see Appendix~\ref{sec:total_complexity}.
\item[\bf 8:] The heuristic choice of $\mu_k$ in (\ref{eq:mu_k_estimate}).
\item[\bf 10:] The restart procedure for inadequate estimates of $\mu$.
\end{description}

We note that in a practical implementation, the computational work
involved in one iteration step of \algname{} may -- in the worst case
situation -- be twice that of one iteration of \algnameBB, due to the
two calls to \algnameBT{}.  However, it may be
possible to implement these two calls more efficiently than naively
calling \algnameBT{} twice. We will instead focus on the iteration
complexity of \algname{} given in the following theorem.

\begin{theorem} \label{theo:upn}
Algorithm \algname{}, applied to $f \in \FmuL$ under conditions
$\bar\mu\geq\mu$, $\bar L\leq L$, $\bar{\epsilon} = \sqrt{(\mu/2) \,\epsilon}$, stops using the gradient map magnitude measure and returns an $\epsilon$-suboptimal
solution with iteration complexity
  \begin{align}\label{eq:complexity_stop}
    \OC\Bigl( \sqrt{Q}\log Q \Bigr) + \OC\left(
    \sqrt{Q}\log \epsilon^{-1}  \right ) .
  \end{align}
\end{theorem}

\begin{proof}
See Appendix \ref{sec:complexity}.
\hfill $\Box$
\end{proof}

The term $\OC\left( \sqrt{Q}\log Q \right)$ in (\ref{eq:complexity_stop}) follows
from application of several inequalities involving the problem dependent parameters
$\mu$ and $L$ to obtain the overall bound (\ref{eq:iteration_mu_sufficiently_small}).
Algorithm \algname{} is suboptimal since the optimal complexity is
$\OC \left( \sqrt{Q}\log \epsilon^{-1} \right)$ but it has the advantage that
it can be applied to problems with unknown $\mu$ and $L$.
%

\section{The 3D Tomography Test Problem} \label{sec:tomo}

Tomography problems arise in numerous areas, such as medical imaging,
non-destruc\-tive testing, materials science, and geophysics
\cite{Herman,KaSl01,Nolet}.
These problems amount to reconstructing an object from its projections
along a number of specified directions, and these projections are produced
by X-rays, seismic waves, or other ``rays'' penetrating the object
in such a way that their intensity is partially absorbed by the object.
The absorbtion thus gives information about the object.

The following generic model accounts for several applications of tomography.
We consider an object in 3D with linear attenuation coefficient $\Xfun(t)$, with
$t\in\mathrm{\Omega}\subset\mathbb{R}^3$.
The intensity decay $b_i$ of a ray along the line $\ell_i$ through
$\mathrm{\Omega}$ is governed by a line integral,
  \begin{equation}
    b_i = \log (I_0/I_i) = \int_{\ell_i} \Xfun(t) \, d\ell = b_i,
  \end{equation}
where $I_0$ and $I_i$ are the intensities of the ray before and after passing
through the object.
When a large number of these line integrals are recorded, then we are able to
reconstruct an approximation of the function~$\Xfun(t)$.

We discretize the problem as described in Section~\ref{sec:TVproblem}, such
that $\Xfun$ is approximated by a piecewise constant function in each voxel
in the domain $\mathrm{\Omega} = [0,1] \times [0,1] \times [0,1]$.
Then the line integral along $\ell_i$ is computed by summing the contributions
from all the voxels penetrated by $\ell_i$.
If the path length of the $i$th ray through the $j$th voxel is denoted
by~$a_{ij}$, then we obtain the linear equations
  \begin{equation}
    \sum_{j=1}^{\vmax} a_{ij} x_j = b_i, \qquad i = 1,\dots, \rmax ,
  \end{equation}
where $\rmax$ is the number of rays or measurements and $\vmax$ is the number of voxels.
This is a linear system of equations~$A\, x = b$ with a
sparse coefficient matrix $A\in\mathbb{R}^{\rmax\times \vmax}$.

A widely used test image in medical tomography is the ``Shepp-Logan phantom,''
which consists of a number superimposed ellipses.
In the MATLAB function \texttt{shepplogan3d} \cite{shepplogan3d}
this 2D image is generalized to 3D by superimposing ellipsoids instead.
The voxels are in the range $[0,1]$, and \fig \ref{fig:phantom3d} shows
an example with $43 \times 43 \times 43$ voxels.

\begin{figure}
\begin{center}
\ifpdf
	\includegraphics[width=0.3\textwidth]{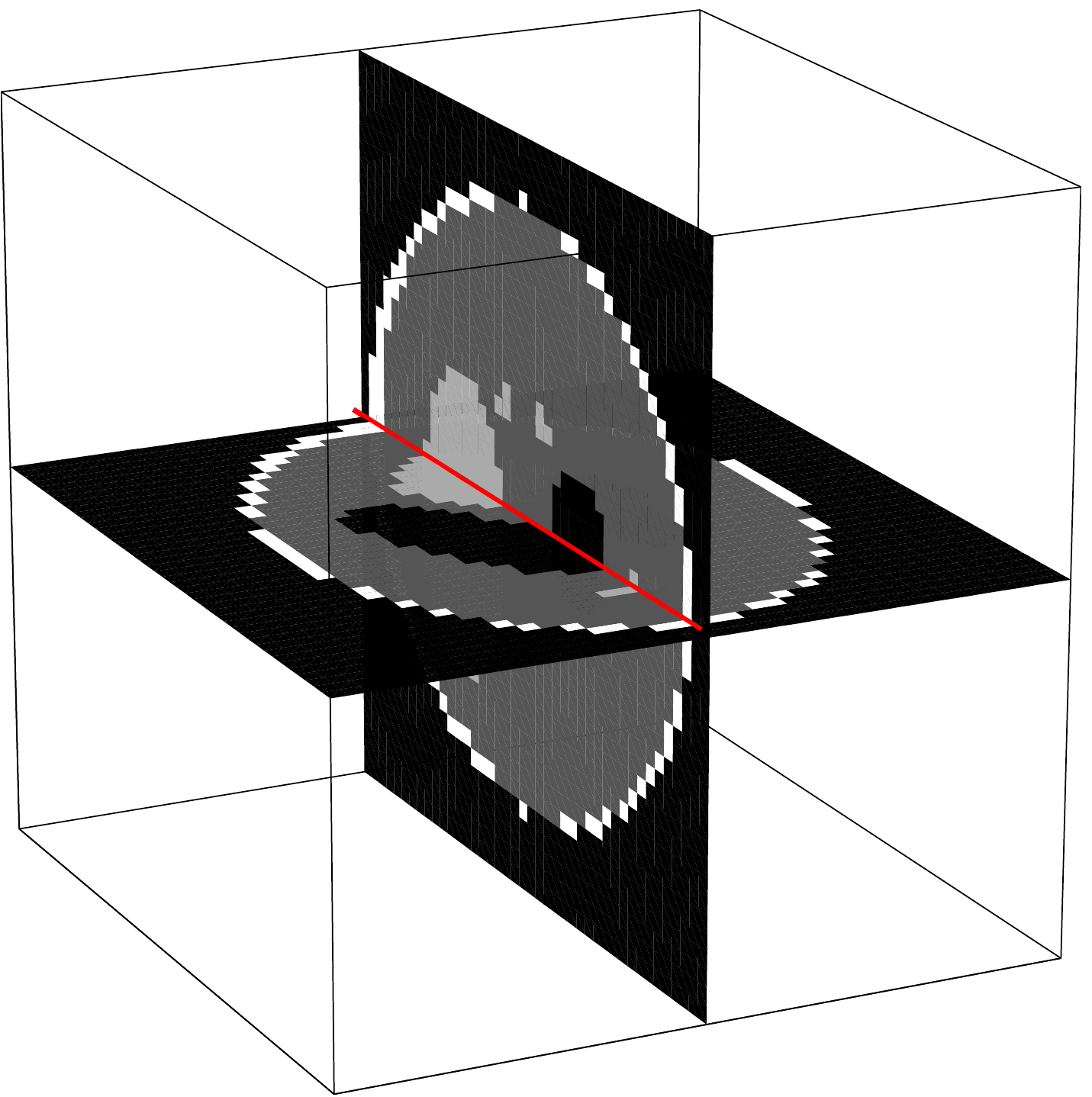}
	\hspace{0.03\textwidth}
	\includegraphics[width=0.3\textwidth]{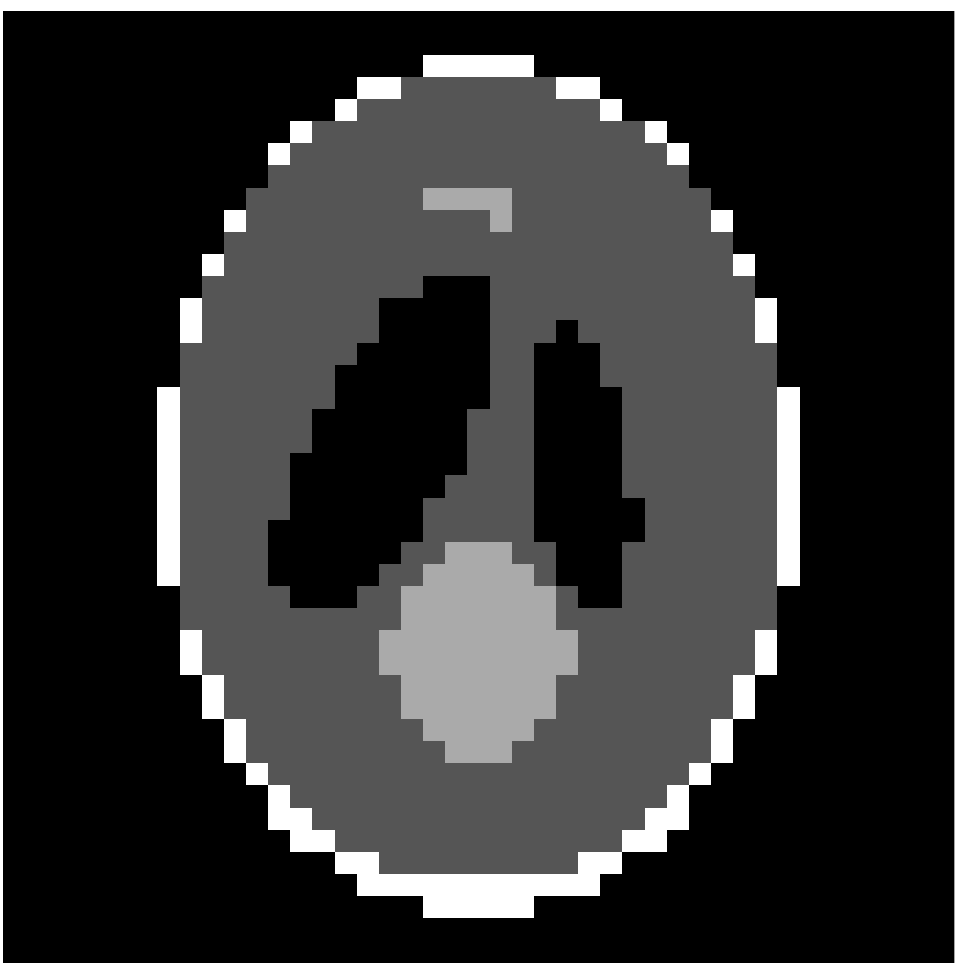}
	\hspace{0.03\textwidth}
	\includegraphics[width=0.3\textwidth]{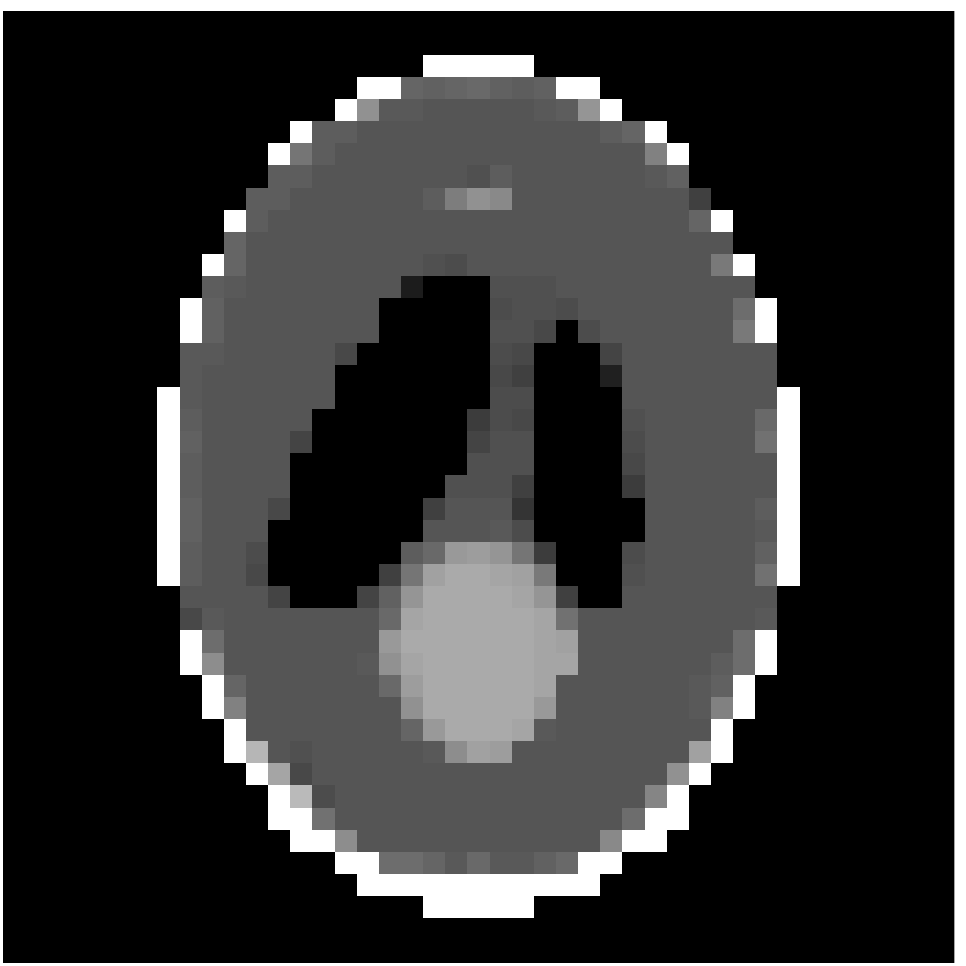}
\else
	\includegraphics[width=0.3\textwidth]{single_rec_paper_3dbw.eps}
	\hspace{0.03\textwidth}
	\includegraphics[width=0.3\textwidth]{single_rec_paper_orig.eps}
	\hspace{0.03\textwidth}
	\includegraphics[width=0.3\textwidth]{single_rec_paper_rec.eps}
\fi
\end{center}
\caption{Left: Two orthogonal slices through the 3D Shepp-Logan phantom
discretized on a $43^3$ grid used in our test problems. Middle: Central horizontal slice.  Right: Example of solution for $\alpha=1$ and $\tau = 10^{-4}$. A less smooth solution can be obtained using a smaller $\alpha$. Original voxel/pixel values are $0.0$, $0.2$, $0.3$ and $1.0$. Color range in display is set to $[0.1,0.4]$ for better constrast.}
\label{fig:phantom3d}
\end{figure}


We construct the matrix $A$ for a parallel-beam geometry with orthogonal projections
of the object along directions well distributed over the unit sphere,
in order to obtain views of the object that are as independent as possible.
The projection directions are the direction vectors of so-called
\emph{Lebedev quadrature} points on the unit sphere, and the directions are
evenly distributed over the sphere;
we use the MATLAB implementation \texttt{getLebedevSphere}~\cite{lebedevquad}.
For setting up the tomography system matrix for a parallel beam geometry, we use 
the Matlab implementation \texttt{tomobox}~\cite{tomobox}.

\section{Numerical Experiments} \label{sec:numerical}
This section documents numerical experiments with the four methods \algname{}, \algnameUPNz{}, \algnameGP{} and \algnameBB{} applied to the TV regularization problem \eqref{eq:discreteTV}.
We use the two test problems listed in Table \ref{tab:tesstcases},
which are representative across a larger class of problems (other
directions, number of projections, noise levels, etc.) that we have
run simulations with.  The smallest
eigenvalue of $A^TA$ for \Tover{} is $2.19\cdot 10^{-5}$ (as computed by
\textsc{Matlab}'s \texttt{eigs}), confirming that $\rank(A)=N$ for \Tover{}. We emphasize
that this computation is only conducted to support the analysis of the
considered problems since -- as we have argued in the introduction -- it
carries a considerable computational burden to compute. In all
simulations we create noisy data from an exact object $x_\mathrm{exact}$ through the forward mapping $b = A x_\mathrm{exact} + e$, subject to additive Gaussian white noise of relative noise level $\|e\|_2/\|b\|_2 = 0.01$. As
initial vector for the TV algorithms we use the fifth iteration of the
iterative conjugate gradient method applied to the least squares problem.

\begin{table}[htbp]
\vspace{-0.5cm}
\caption{Specifications of the two test problems; the object domain consists
of $m \times n \times l$ voxels and each projection is a $p\times p$ image.
Any zero rows have been purged from~$A$. }
\centering
\begin{tabular}{c|c|c|c|c|c}
\hline
Problem & $m=n=l$ & $p$ & projections & dimensions of $A$ & rank\\ \hline
\Tover{} & 43 & 63 & 37 & $ 99361\times 79507$ & $= 79507$ \\
\Tunder{} & 43 & 63 & 13 & $ 33937\times 79507$ & $< 79507$ \\
 \hline
\end{tabular}
\label{tab:tesstcases}
\vspace{-0.2cm}
\end{table}

To investigate the convergence of the methods, we need the true minimizer
$x^\star$ with $\phi(x^\star) = \phi^\star$, which is unknown for the
test problem.
However, for comparison it is enough to use a reference solution much closer to the true minimizer than the iterates.
Thus, to compare the accuracy of the solutions obtained with the accuracy
parameter $\bar \epsilon$, we use a reference solution computed with
accuracy $(\bar \epsilon\cdot 10^{-4})$, and with abuse of
notation we use $x^\star$ to denote this reference solution.

We compare the algorithm \algname\ with \algnameGP\ (the gradient
projection method \eqref{eq:gradientmethod} with backtracking line search on the step size), \algnameBB\ and \algnameUPNz.
The latter is \algname\ with $\mu_i=0$ for all $i=0,\cdots,k$ and
$\theta_1=1$. The algorithm \algnameUPNz\ is optimal for the class $\FzL$ and
can be seen as an instance of the more general accelerated/fast proximal
gradient algorithm (FISTA) with backtracking and the non-smooth term being the
indicator function for the set $\Qset$, see \cite{beck:2009,beck:2009b,Ne:07} and the overview in \cite{Tse:08}. 

\subsection{Influence of $\alpha$ and $\tau$ on the convergence}
For a given $A$ the theoretical modulus of strong convexity given in
\eqref{eq:Qf} varies only with $\alpha$ and $\tau$. We therefore expect
better convergence rates \eqref{eq:gradientrate1} and
\eqref{eq:nesterovrate_2} for smaller $\alpha$ and larger $\tau$. In
\fig \ref{fig:alphatau} we show the convergence histories for
\Tover{} with all combinations of $\alpha = 0.01$, $0.1$, $1$ and $\tau = 10^{-2}$, $10^{-4}$,  $10^{-6}$.

\begin{figure}[tbp]
\ifpdf
    \includegraphics[width=1\textwidth]{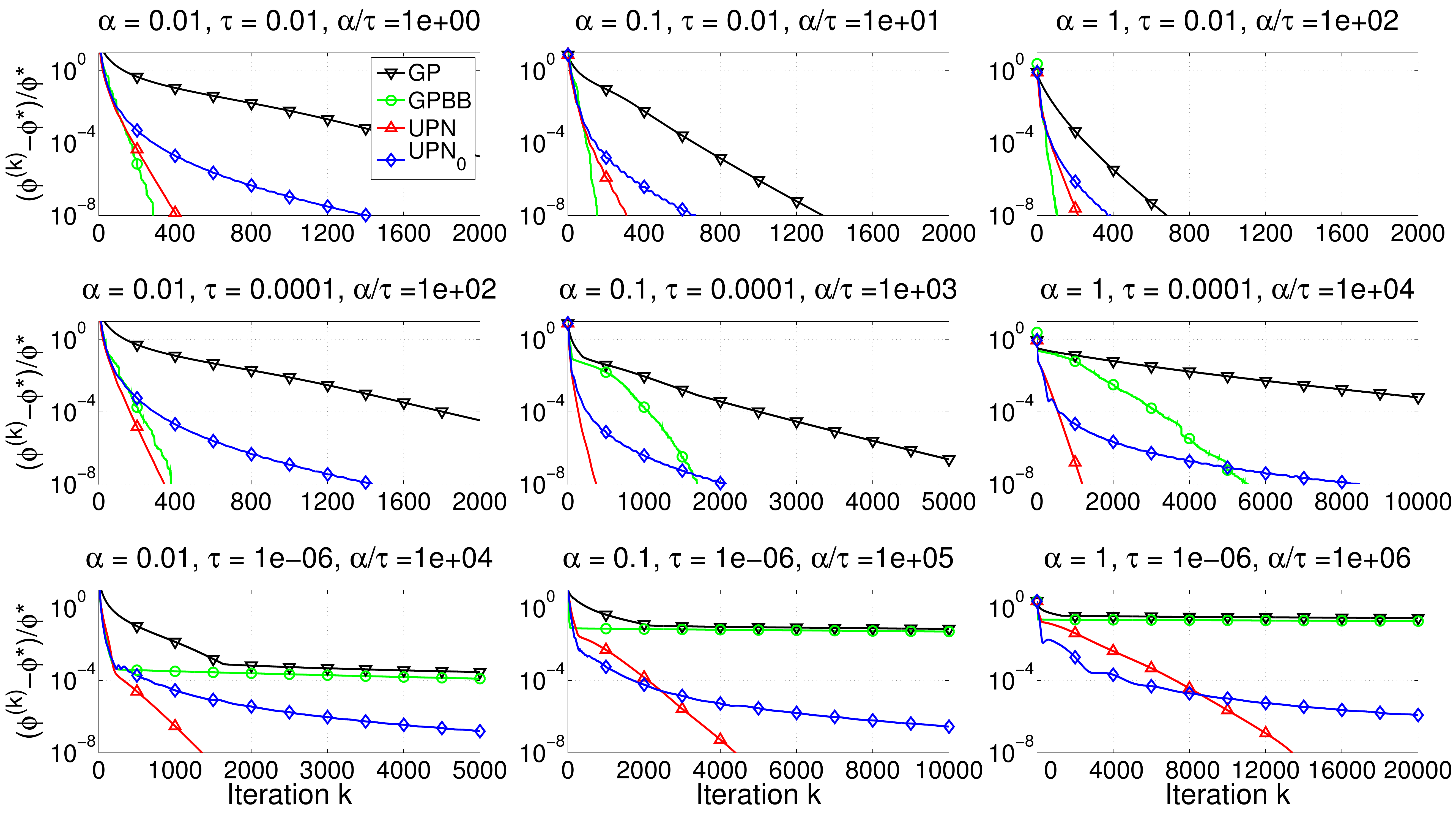}
\else
    \includegraphics[width=1\textwidth]{test001_con_conv.eps}
\fi
\caption{Convergence histories $(\phi(x^{(k)})-\phi^\star)/\phi^\star$ vs. $k$ for \Tover{} with $\alpha = 0.01$, $0.1$ and $1$ and $\tau = 10^{-2}$, $10^{-4}$ and $10^{-6}$.} \label{fig:alphatau}
\end{figure}

For low $\alpha/\tau$ ratios, i.e., small condition number of the
Hessian, \algnameBB{} and \algnameGP{} requires a comparable or smaller
number of iterations than \algname{} and \algnameUPNz{}. As
$\alpha/\tau$ increases, both \algnameBB{} and
\algnameGP{} exhibit slower convergence, while \algname{} is less affected. In
all cases \algname{} shows linear convergence, at least in the final
stage, while \algnameUPNz{} shows sublinear convergence. Due to these
observations, we consistently observe that for sufficiently high
accuracy, \algname{} requires the lowest number of iterations. This
also follows from the theory since \algname{} scales as $\OC ( \log
\epsilon^{-1})$, whereas \algnameUPNz{} scales at a
higher complexity of $\OC (\sqrt{\epsilon^{-1}})$.

We conclude that for small condition numbers there is no gain in using \algname{} compared to \algnameBB{}. For larger condition numbers, and in particular if a high-accuracy solution is required, \algname{} converges significantly faster. Assume that we were to choose only one of the four algorithms to use for
reconstruction across the condition number range. When \algname{}
requires the lowest number of iterations, it requires \emph{signi\-ficantly}
fewer, and when not, \algname\ only requires slightly more iterations
than the best of the other algorithms. Therefore, \algname{} appears to be the best
choice. Obviously, the choice of algorithm also depends on the demanded accuracy of the
solution. If only a low accuracy, say
$(\phi^{(k)}-\phi^\star)/\phi^\star = 10^{-2}$ is sufficient, all four
methods perform more or less equally well.

\subsection{Restarts and $\mu_k$ and $L_k$ histories}
To ensure convergence of \algname{} we introduced the restart
functionality \algnameR{}. In practice, we almost never observe a
restart, e.g., in none of the experiments reported so far a restart
occurred. An example where restarts do occur is obtained if we
increase $\alpha$ to $100$ for \Tover{} (still $\tau =
10^{-4}$). Restarts occur in the first $8$ iterations, and each time
$\mu_k$ is reduced by a constant factor of $\rho_\mu=0.7$. In \fig
\ref{fig:restart}, left, the $\mu_k$ and $L_k$ histories are plotted
vs. $k$ and the restarts are seen in the zoomed inset as the rapid,
constant decrease in $\mu_k$.
\begin{figure}[tbp]
\ifpdf
    \includegraphics[width=0.51\textwidth]{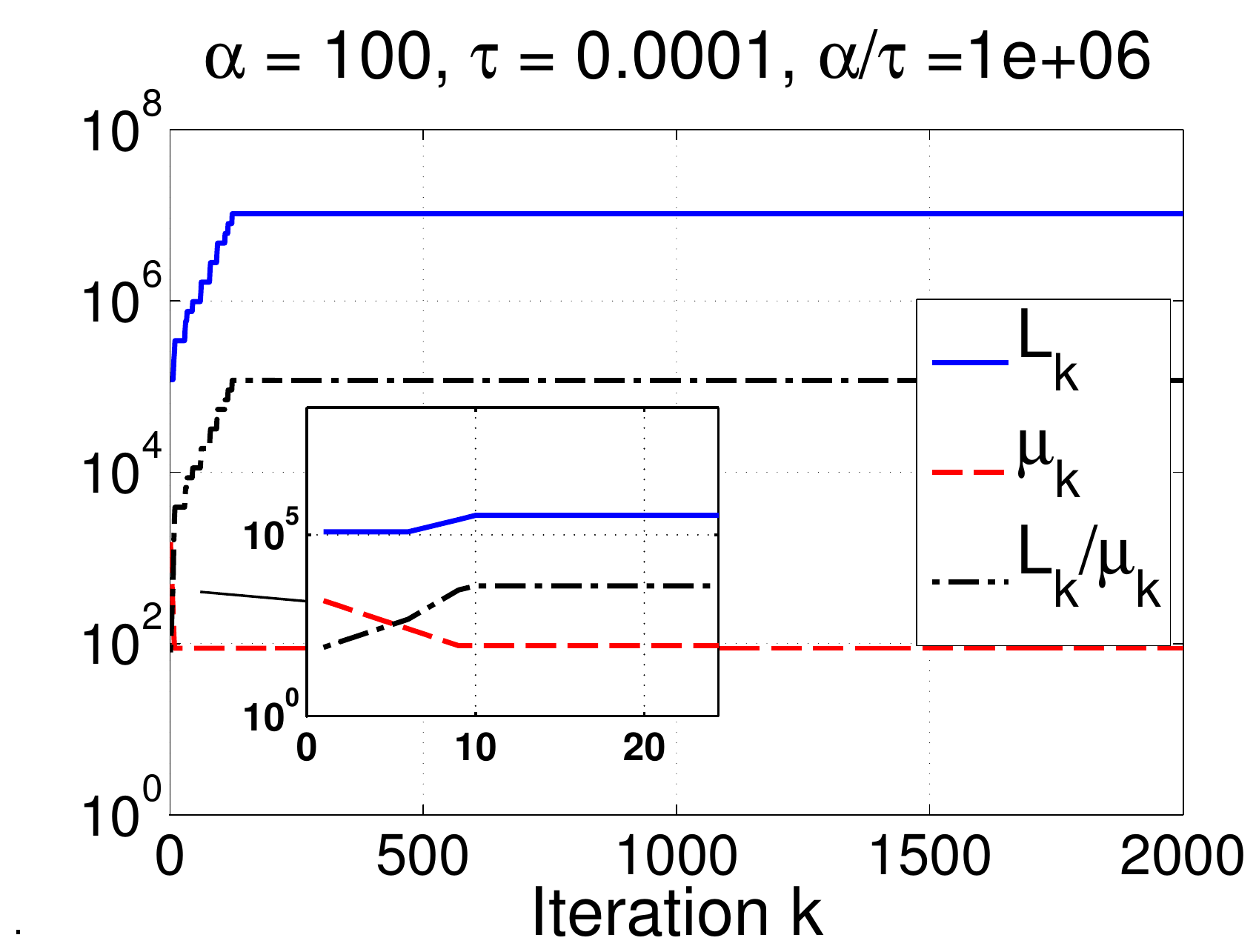}
    \includegraphics[width=0.48\textwidth]{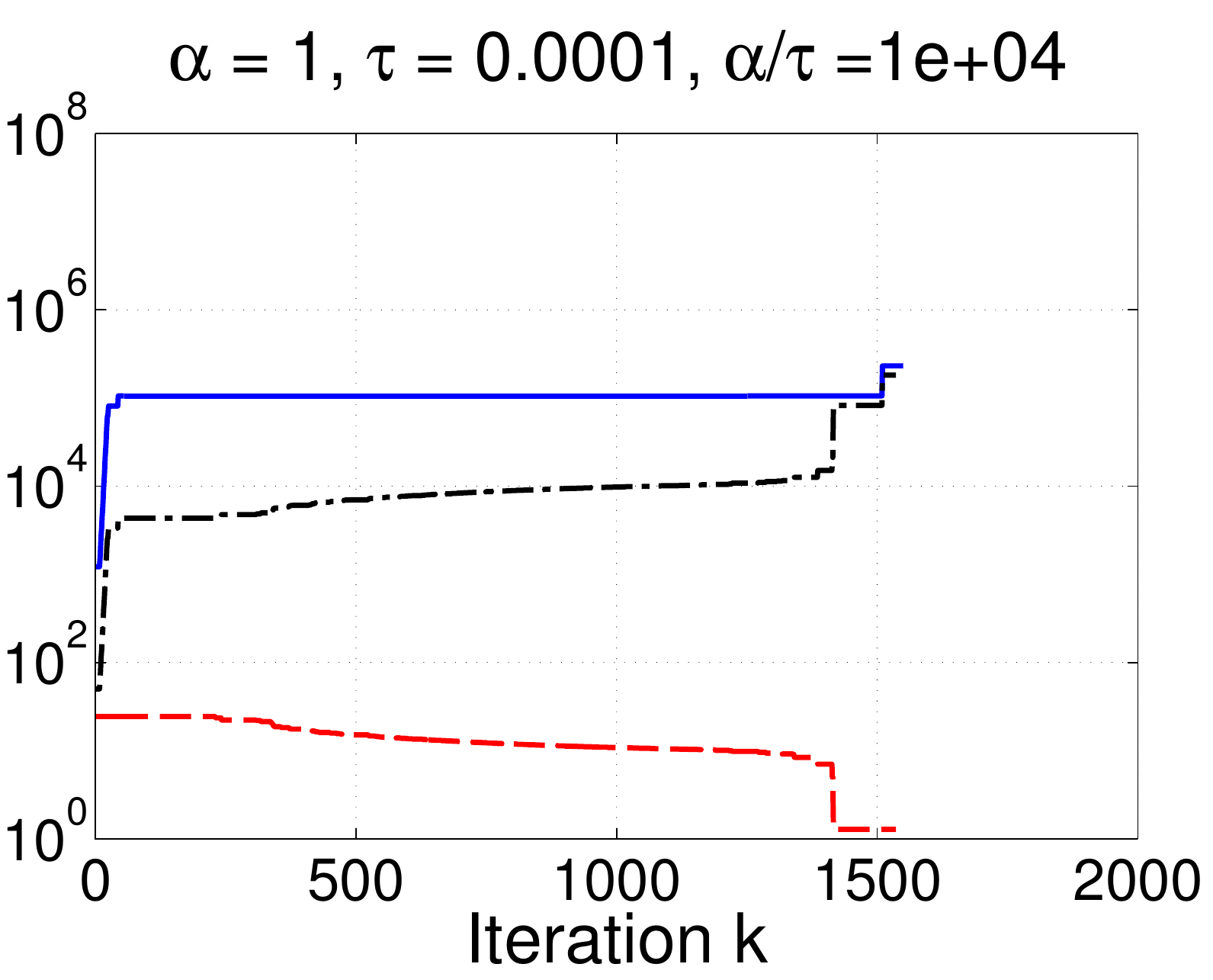}
\else
    \includegraphics[width=0.51\textwidth]{restart_inset.eps}
    \includegraphics[width=0.48\textwidth]{mukLk_001_single.eps}
\fi
\caption{The $\mu_k$, $L_k$ histories for \Tover{}. Left: $\alpha = 100$ and $\tau = 10^{-4}$. Right: $\alpha = 1$ and $\tau = 10^{-4}$.} \label{fig:restart}
\end{figure}
From the plot we also note that after the decrease in $\mu_k$ and an
initial increase in $L_k$, both estimates are constant for the
remaining iterations, indicating that the heuristics determines
sufficient values. 

For comparison the $\mu_k$ and $L_k$ histories for
\Tover{} with $\alpha = 1$ and $\tau = 10^{-4}$ are seen in \fig
\ref{fig:restart}, right. No restarts occurred here, and $\mu_k$ decays
gradually, except for one final jump, while $L_k$ remains almost constant.

\subsection{A non-strongly convex example}
Test problem \Tunder{} corresponds to only 13 projections, which
causes $A$ to not have full column rank. This leads to
$\lambda_\mathrm{min}(A^T A) = 0$, and hence $\phi(x)$ is not strongly
convex. The optimal convergence rate is therefore given by 
\eqref{eq:nesterovrate}; but how does the lack of strong convexity
affect \algname{}, which was specifically constructed for strongly
convex problems? \algname{} does not recognize that the problem is not
strongly convex but simply relies on the heuristic
\eqref{eq:mu_k_estimate} at the $k$th iteration. We investigate the
convergence by solving \Tunder{} with $\alpha=1$ and $\tau=10^{-4}$. Convergence histories are given in \fig \ref{fig:underdetermined_true_mu_L}, left.
%
The algorithm \algname{} still converges linearly, although slightly slower than in the \Tover{} experiment
($\alpha=1,\tau=10^{-4}$) in \fig \ref{fig:alphatau}. The
algorithms \algnameGP{} and \algnameBB{} converge much more slowly,
while at low accuracies \algnameUPNz{} is comparable to \algname. But the linear convergence
makes \algname{} converge faster for high accuracy solutions.

\subsection{Influence of the heuristic}
An obvious question is how the use of the heuristic for estimating
$\mu$ affects \algname{} compared to \algnameNesterov{}, where
$\mu$ (and $L$) are assumed known. From Theorem \ref{thm:Qtight} we can
compute a strong convexity parameter and a Lipschitz parameter for
$\phi(x)$ assuming we know the largest and smallest magnitude
eigenvalues of $A^TA$. Recall that these $\mu$ and $L$ are not necessarily the tightest possible, according to Remark \ref{remark:nontight}. For \Tover{} we have computed
$\lambda_\mathrm{max}(A^TA) = 1.52\cdot 10^3$ and
$\lambda_\mathrm{min}(A^T A) = 2.19\cdot 10^{-5}$ (by means of \texttt{eigs} in \textsc{Matlab}). Using $\alpha = 1$, $\tau = 10^{-4}$ and $\|D\|_2^2 \leq 12$ from Lemma \ref{lem:TVmuL} we take
\begin{align*}
 \mu = \lambda_\mathrm{min}(A^T A) = 2.19\cdot 10^{-5}, \qquad 
 L = \lambda_\mathrm{max}(A^T A)  + 12 \frac{\alpha}{\tau} = 1.22 \cdot 10^5,
\end{align*}
and solve test problem \Tover{} using \algname{} with the heuristics switched off in favor of these \emph{true} strong convexity and Lipschitz parameters. Convergence histories are plotted in \hbox{\fig \ref{fig:underdetermined_true_mu_L}, right}.
\begin{figure}[tbp]
\centering
\ifpdf
    \includegraphics[width=0.49\textwidth]{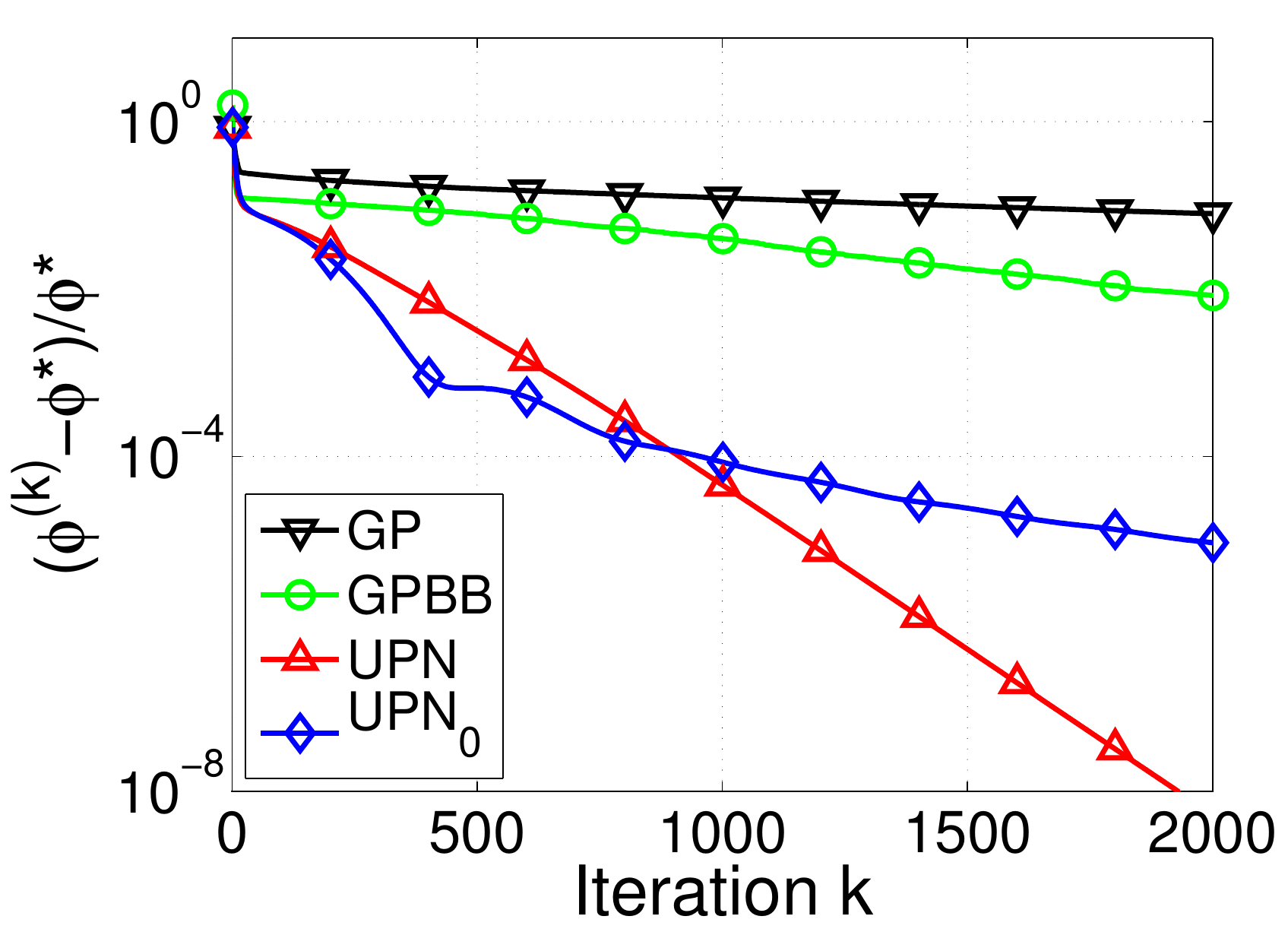}
    \includegraphics[width=0.49\textwidth]{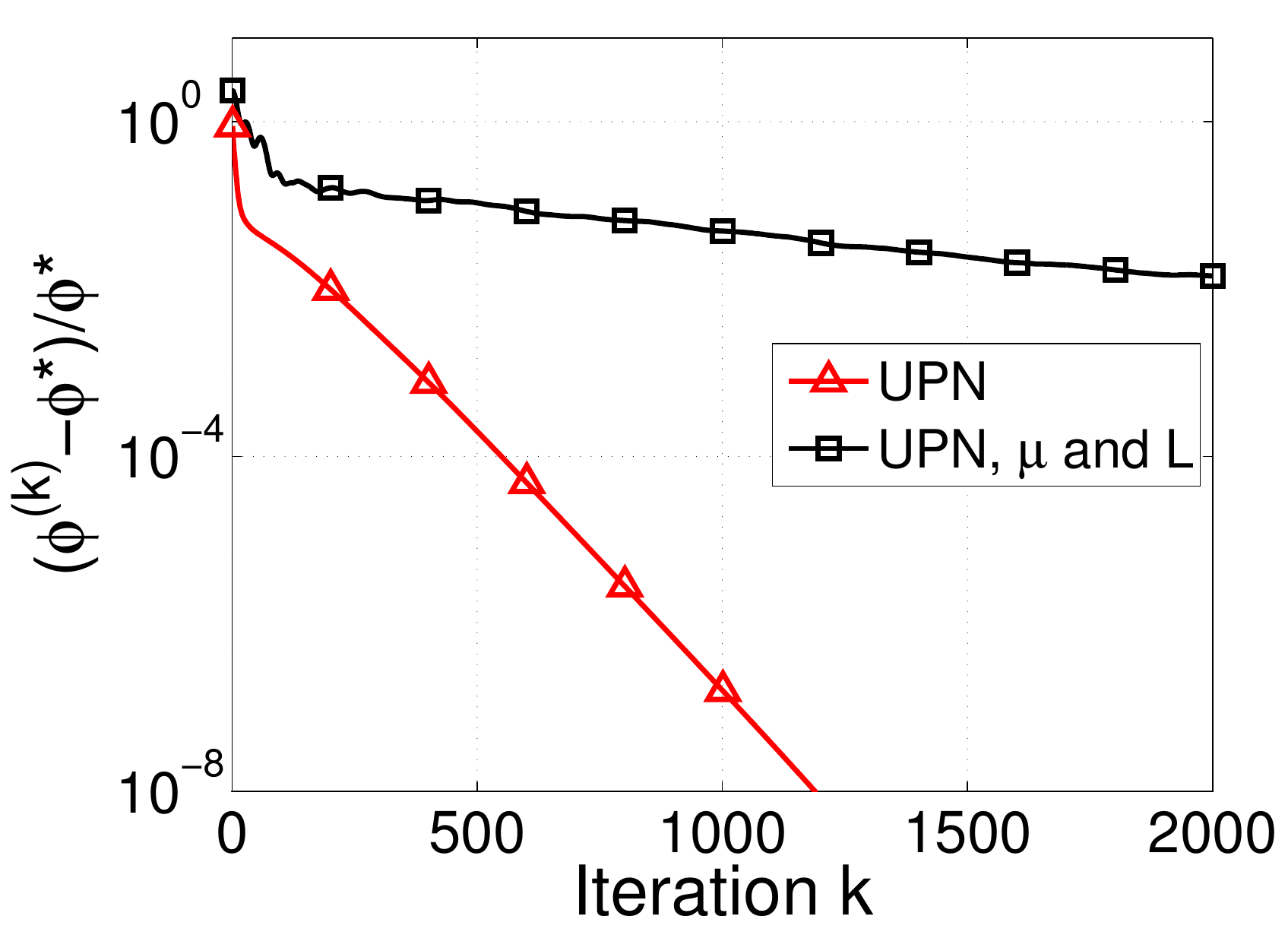}
\else
    \includegraphics[width=0.49\textwidth]{alpha1tau1e-4numproj13_conv.eps}
    \includegraphics[width=0.49\textwidth]{true_mu_L.eps}
\fi
\caption{Left: Convergence histories of \algnameGP{}, \algnameBB{}, \algname{} and \algnameUPNz{} on \Tunder{} with $\alpha = 1$ and $\tau= 10^{-4}$. Right: Convergence histories of \algname{} and \algname{} using true $\mu$ and $L$ on \Tover{} with $\alpha = 1$ and $\tau = 10^{-4}$.} \label{fig:underdetermined_true_mu_L}
\end{figure}
%

The convergence is much slower than using \algname{} with the heuristics switched on. We ascribe this behavior to the very large modulus of strong
convexity that arise from the true $\mu$ and $L$. It appears that \algname{} works
better than the actual degree of strong convexity as measured by
$\mu$, by heuristically choosing in each step a $\mu_k$ that is sufficient
\emph{locally} instead of being restricted to using a \emph{globally}
\hbox{valid $\mu$.}

\section{Conclusion} \label{sec:conclusion}
We presented an implementation of an optimal first-order optimization
algorithm for large-scale problems, suited for functions that are
smooth and strongly convex.
While the underlying algorithm by Nesterov depends on knowledge of two
parameters that characterize the smoothness and strong convexity, we have
implemented methods that estimate these parameters during the
iterations, thus making the algorithm of practical use.

We tested the performance of the algorithm and compared it with two variants of the gradient projection algorithm and a variant of the FISTA algorithm. We applied the algorithms to total variation-regularized tomographic reconstruction of a generic threedimensional test problem.
The tests show that, with regards to the number of ite\-rations, the
proposed algorithm is competitive with other first-order algorithms,
and superior for difficult problems, i.e., ill-conditioned problems solved
to high accuracy. Simulations also show that even for problems that
are not strongly convex, in practice we achieve the favorable
convergence rate associated with strong complexity. The software is
available as a C-implementation with an interface to MATLAB from
\url{www2.imm.dtu.dk/~pch/TVReg/}.


\appendix

\section{The Optimal Convergence Rate} \label{app:optimal_convergence_rate}

Here we provide an analysis of an optimal method for smooth, strongly
convex functions without the use of estimation functions as in~\cite{Ne:04}.
This approach is similar to the analysis of optimal methods for smooth functions
in \cite{Tse:08,Va:09}.
The motivation for the following derivations is to introduce the iteration
dependent $L_k$ and $\mu_k$ estimates of $L$ and $\mu$.
This will support the analysis of how $L_k$ and $\mu_k$ should be selected.
We start with the following relations to the ``hidden'' supporting variables
$z^{(k)}$ and $\gamma_k$ \cite[pp.~73--75, 89]{Ne:04},
  \begin{equation}
  \label{eqn:relation_1}
    y^{(k)} - x^{(k)} = \frac{\theta_k \gamma_k}{\gamma_{k+1}} (z^{(k)}-y^{(k)}) ,
  \end{equation}
  \begin{equation}
  \label{eqn:relation_2}
    \gamma_{k+1} = (1-\theta_k)\gamma_k + \theta_k \mu_k = \theta_k^2 L_k ,
    \qquad
    \gamma_{k+1} z^{(k+1)} = (1-\theta_k)\gamma_k z^{(k)} +
    \theta_k \mu_k y^{(k)} - \theta_k G_{L_k}(y^{(k)}) .
  \end{equation}
In addition we will make use of the relations
  \begin{eqnarray}
  \nonumber
    \frac{\gamma_{k+1}}{2} \| z^{(k+1)} - y^{(k)} \|_2^2 & = &
      \frac{1}{2\gamma_{k+1}} \Big( (1-\theta_k)^2 \gamma_k^2 \|z^{(k)} -y^{(k)}\|_2^2 \\
  \label{eqn:relation_4}
    & & - 2 \theta_k (1-\theta_k)\gamma_k G_{L_{k}}(y^{(k)})^T(z^{(k)}-y^{(k)})
    + \theta_k^2 \|G_{L_{k}}(y^{(k)})) \|_2^2 \Big) ,
  \end{eqnarray}
  \begin{equation}
  \label{eqn:relation_5}
    (1-\theta_k)\frac{\gamma_k}{2}-\frac{1}{2\gamma_{k+1}}(1-\theta_k)^2\gamma_k^2
    = \frac{(1-\theta_k) \gamma_k \theta_k \mu_k}{2\gamma_{k+1}} .
  \end{equation}
which originate from \eqref{eqn:relation_2}. We will also later need
the relation
  \begin{eqnarray}
  \nonumber
    \lefteqn{ (1-\theta_k) \frac{\gamma_k}{2} \| z^{(k)}-y^{(k)} \|_2^2
    - \frac{\gamma_{k+1}}{2} \| z^{(k+1)}-y^{(k)} \|_2^2 + \theta_k
    G_{L_k}(y^{(k)})^T(y^{(k)}-x^\star) } \\
  \nonumber
    & = & (1-\theta_k) \frac{\gamma_k}{2} \| z^{(k)}-y^{(k)} \|_2^2
    - \frac{\gamma_{k+1}}{2} \| z^{(k+1)}-y^{(k)} \|_2^2
    + \left( -\gamma_{k+1} z^{(k+1)} + (1-\theta_k) \gamma_k z^{(k)} +
    \theta_k \mu_k y^{(k)} \right)^T (y^{(k)}-x^\star) \\
  \nonumber
    & = & \left( (1-\theta_k) \frac{\gamma_k}{2} - \frac{\gamma_{k+1}}{2} +
    \theta_k \mu_k \right) (y^{(k)})^T y^{(k)} +
    (1-\theta_k) \frac{\gamma_k}{2} (z^{(k)})^T z^{(k)} - \frac{\gamma_{k+1}}{2}
    (z^{(k+1)})^T z^{(k+1)} \\
  \nonumber
    & & + \ \gamma_{k+1}(z^{(k+1)})^T x^\star - (1-\theta_k) \gamma_k (z^{(k)})^T x^\star
    - \theta_k \mu_k (y^{(k)})^T x^\star \\
  \nonumber
    & = & (1-\theta_k) \frac{\gamma_k}{2} \left( \| z^{(k)}-x^\star \|_2^2
    - (x^\star)^T x^\star \right) - \frac{\gamma_{k+1}}{2} \left( \| z^{(k+1)}-x^\star \|_2^2
    - (x^\star)^T x^\star \right) \\
  \nonumber
    & & + \ \frac{\theta_k \mu_k}{2} \left( \| y^{(k)}-x^\star \|_2^2 -
    (x^\star)^T x^\star \right) + \left(
    (1-\theta_k) \frac{\gamma_k}{2} - \frac{\gamma_{k+1}}{2} +
    \frac{\theta_k\mu_k}{2} \right) (y^{(k)})^T y^{(k)} \\
  \label{eqn:relation_6}
    & = & (1+\theta_k) \frac{\gamma_k}{2} \| z^{(k)}-x^\star \|_2^2
    - \frac{\gamma_{k+1}}{2} \| z^{(k)}-x^\star \|_2^2 +
    \theta_k \frac{\mu_k}{2} \| y^{(k)}-x^\star \|_2^2 ,
  \end{eqnarray}
where we again  used \eqref{eqn:relation_2}.
We can now start the analysis of the algorithm by considering the
inequality in Lemma \ref{lemma_gradient_map},
  \begin{equation}\label{eq:osc_ine1}
    (1-\theta_k) f(x^{(k+1)}) \leq (1-\theta_k) f(x^{(k)})
    + (1-\theta_k) G_{L_k}(y^{(k)})^T (y^{(k)}-x^{(k)})
    - (1-\theta_k) \frac{1}{2L_k} \| G_{L_k}(y^{(k)}) \|_2^2 ,
  \end{equation}
where we have omitted the strong convexity part, and the inequality
  \begin{equation}\label{eq:osc_ine2}
    \theta_k f(x^{(k+1)}) \leq \theta_k f(x^\star)
    + \theta_k G_{L_k}(y^{(k)})^T (y^{(k)}-x^\star)
    - \theta_k \frac{1}{2L_k} \| G_{L_k}(y^{(k)}) \|_2^2
    - \theta_k \frac{\mu_k^\star}{2} \| y^{(k)}-x^\star \|_2^2 .
  \end{equation}
Adding these bounds and continuing, we obtain
  \begin{eqnarray}
  \nonumber
    f(x^{(k+1)}) & \leq & (1-\theta_k) f(x^{(k)}) + (1-\theta_k) G_{L_k}(y^{(k)})^T
    (y^{(k)}-x^{(k)}) \\
  \nonumber
    & & + \ \theta_k f^\star + \theta_k G_{L_k}(y^{(k)})^T(y^{(k)}-x^\star) - \theta_k
    \frac{\mu_k^\star}{2} \| x^\star-y^{(k)} \|_2^2 - \frac{1}{2 L_k} \| G_{L_k}(y^{(k)}) \|_2^2 \\
  \nonumber
    & = & (1-\theta_k) f(x^{(k)}) + (1-\theta_k) \frac{\theta_k\gamma_k}{\gamma_{k+1}}
    G_{L_k}(y^{(k)})^T (z^{(k)}-y^{(k)}) \\
  \nonumber
    & & + \ \theta_k f^\star + \theta_k G_{L_k}(y^{(k)})^T (y^{(k)}-x^\star) - \theta_k
    \frac{\mu_k^\star}{2} \| x^\star-y^{(k)} \|_2^2 - \frac{1}{2 L_k} \| G_{L_k}(y^{(k)})\|_2^2 \\
  \nonumber
    & \leq & (1-\theta_k) f(x^{(k)}) + (1-\theta_k) \frac{\theta_k\gamma_k}{\gamma_{k+1}}
    G_{L_k}(y^{(k)})^T (z^{(k)}-y^{(k)}) \\
  \nonumber
    & & + \ \theta_k f^\star + \theta_k G_{L_k}(y^{(k)})^T (y^{(k)}-x^\star) - \theta_k
    \frac{\mu_k^\star}{2} \| x^\star - y^{(k)} \|_2^2 - \frac{1}{2 L_k} \|G_{L_k}(y^{(k)})\|_2^2 \\
  \nonumber
    & & + \ \frac{(1-\theta_k) \theta_k \gamma_k \mu_k}{2 \gamma_{k+1} }
    \| z^{(k)}-y^{(k)}\|_2^2 \\
  \nonumber
    & = & (1-\theta_k) f(x^{(k)}) + (1-\theta_k) \frac{\theta_k\gamma_k}{\gamma_{k+1}}
    G_{L_k}(y^{(k)})^T (z^{(k)}-y^{(k)}) \\
  \nonumber
    & & + \ \theta_k f^\star + \theta_k G_{L_k}(y^{(k)})^T (y^{(k)}-x^\star) - \theta_k
    \frac{\mu_k^\star}{2} \| x^\star-y^{(k)} \|_2^2 - \frac{1}{2 L_k} \|G_{L_k}(y^{(k)}) \|_2^2 \\
  \nonumber
    & & + \ \left( (1-\theta_k) \frac{\gamma_k}{2}-\frac{1}{2\gamma_{k+1}}
    (1-\theta_k)^2 \gamma_{k}^2 \right) \| z^{(k)}-y^{(k)} \|_2^2 \\
  \nonumber
    & = & (1-\theta_k) f(x^{(k)}) + (1-\theta_k) \frac{\gamma_k}{2} \| z^{(k)}-y^{(k)} \|_2^2
    - \frac{\gamma_{k+1}}{2} \| z^{(k+1)}-y^{(k)} \|_2^2 \\
  \nonumber
    & & + \ \theta_k f^\star + \theta_k G_{L_k}(y^{(k)})^T (y^{(k)}-x^\star)-\theta_k
    \frac{\mu_k^\star}{2} \| x^\star-y^{(k)} \|_2^2 \\
  \nonumber
    & = & (1-\theta_k) f(x^{(k)}) + \theta_k f^\star - \theta_k
    \frac{\mu_k^\star}{2} \| x^\star-y^{(k)} \|_2^2 \\
  \nonumber
    & & + \ (1-\theta_k) \frac{\gamma_k}{2} \| z^{(k)}-x^\star\|_2^2
    - \frac{\gamma_{k+1}}{2} \| z^{(k+1)}-x^\star \|_2^2 + \theta_k
    \frac{\mu_k}{2} \| y^{(k)}-x^\star \|_2^2 ,
\end{eqnarray}
where we have used \eqref{eqn:relation_1}, a trivial inequality, \eqref{eqn:relation_5}
, \eqref{eqn:relation_4}, \eqref{eqn:relation_2},
and \eqref{eqn:relation_6}. If $\mu_k\leq\mu_k^\star$ then
  \begin{equation}
    \label{eqn:per_iteration_relation}
    f(x^{(k+1)})-f^\star  + \frac{\gamma_{k+1}}{2} \| z^{(k+1)}-x^\star \|_2^2
    \leq \ (1-\theta_k) \left ( f(x^{(k)})-f^\star+ \frac{\gamma_k}{2} \|
    z^{(k)}-x^\star\|_2^2\right )
  \end{equation}
in which case we can combine the bounds to obtain
  \begin{equation}
    \label{eqn:convergence}
    f(x^{(k)})-f^\star  + \frac{\gamma_{k}}{2} \| z^{(k)}-x^\star \|_2^2
    \leq \ \left(\prod_{i=0}^{k-1}(1-\theta_i) \right)\left
    ( f(x^{(0)})-f^\star+ \frac{\gamma_k}{2} \| z^{(0)}-x^\star\|_2^2\right ),
  \end{equation}
where we have also used $x^{(0)}=y^{(0)}$ and \eqref{eqn:relation_1}
to obtain $x^{(0)}=z^{(0)}$.
For completeness, we will show why this is an optimal first-order method.
Let $\mu_k = \mu_k^\star=\mu$ and $L_k=L$.
If $\gamma_0 \geq \mu$ then using \eqref{eqn:relation_2} we obtain
$\gamma_{k+1} \geq \mu$ and $\theta_k \geq \sqrt{\mu/L}=\sqrt{Q^{-1}}$.
Simultaneously, we also have
$\prod_{i=0}^{k-1} (1-\theta_k)\leq \frac{4L}{\left
    (2\sqrt{L}+k\sqrt{\gamma_0}\right )^2}$ \cite[Lemma 2.2.4]{Ne:04},
and the bound is then
  \begin{equation}
  \label{eq:app_ocr}
    f(x^{(k)})-f^\star \leq \min \left ( \left( 1-\sqrt{Q^{-1}}
      \right)^{\!k}, \frac{4L}{\left (2\sqrt{L}+k\sqrt{\gamma_0}\right)^2}\right)
    \left( f(x^{(0)})-f^\star + \frac{\gamma_0}{2} \|x^{(0)}-x^\star\|_2^2 \right ) .
  \end{equation}
This is the optimal convergence rate for the class $\FzL$ and $\FmuL$
simultaneously \cite{Nemirovsky:1983,Ne:04}.

\section{Complexity Analysis} \label{sec:complexity}

In this Appendix we prove Theorem \ref{theo:upn}, \ie, we derive the complexity for reaching an
$\epsilon$-suboptimal solution for the algorithm $\algname{}$. The total worst-case complexity is given
by a) the complexity for the worst case number of restarts and b) the
worst-case complexity for a successful termination.

With a slight abuse of notation in this Appendix, $\mu_{k,r}$ denotes the
$k$th iterate in the $r$th restart stage, and similarly for
$L_{k,r}$, $\tilde L_{k,r}$, $x^{(k,r)}$, etc.
The value $\mu_{0,0}$ is the initial estimate of the strong convexity
parameter when no restart has occurred.
In the worst case, the heuristic choice in (\ref{eq:mu_k_estimate}) never
reduces $\mu_k$, such that we have $\mu_{k,r} = \mu_{0,r}$.
Then a total of $R$ restarts are required, where
  \begin{align*}
    \rho_{\mu}^R \, \mu_{0,0}=\mu_{0,R} \leq \mu \quad \Longleftrightarrow \quad
    R \geq \log (\mu_{0,0}/\mu)/\log (1/\rho_\mu) .
  \end{align*}
In the following analysis we shall make use of the relation
  \begin{align*}
    \exp \left( -\frac{n}{\delta^{-1}-1}\right )  \leq (1-\delta)^n\leq
    \exp\left (-\frac{n}{\delta^{-1}} \right), \qquad
    0<\delta< 1, \quad n\geq 0 \, .
  \end{align*}

\subsection{Termination Complexity}\label{app:termination_complexity}
After sufficiently many restarts (at most $R$), $\mu_{0,r}$ will be
sufficient small in which case
\eqref{eq:iteration_mu_sufficiently_small} holds and we obtain
  \begin{eqnarray*}
    \| G_{\tilde L_{k+1,r}}(x^{(k+1,r)}) \|_2^2 & \leq &
      \prod_{i=1}^k\left( 1-\sqrt{\frac{\mu_{i,r}}{L_{i,r}}} \right)
      \left( \frac{4\tilde L_{k+1,r}}{\mu_{k,r}} - \frac{2\tilde L_{k+1,r}}{2 L_{0,r}}
      + \frac{2\tilde L_{k+1,r}\gamma_{1,r}}{\mu_{k,r}^2} \right) \| G_{L_{0}}(x^{(0,r)})
      \|_2^2 \\
    & \leq & \left( 1-\sqrt{\frac{\mu_{k,r}}{L_{k,r}}} \right)^k
      \left( \frac{4\tilde L_{k+1,r}}{\mu_{k,r}} - \frac{2\tilde L_{k+1,r}}{2 L_{0,r}}
      + \frac{2\tilde L_{k+1,r}\gamma_{1,r}}{\mu_{k,r}^2} \right) \| G_{L_{0,r}}(x^{(0,r)})
      \|_2^2 \\
    & \leq & \exp\left( -\frac{k}{\sqrt{L_{k,r}/\mu_{k,r}}}\right) \left(
      \frac{4\tilde L_{k+1,r}}{\mu_{k,r}} - \frac{\tilde L_{k+1,r}}{L_{0,r}}
      + \frac{2\tilde L_{k+1,r}\gamma_{1,r}}{\mu_{k,r}^2} \right) \| G_{L_{0,r}}(x^{(0,r)})
      \|_2^2,
  \end{eqnarray*}
where we have used $L_{i,r}\leq L_{i+1,r}$ and $\mu_{i,r}\geq
\mu_{i+1,r}$. To guarantee $\| G_{\tilde L_{k+1,r}}(x^{(k+1,r)}) \|_2\leq \bar \epsilon$
we require the latter bound to be smaller than $\bar \epsilon^2$, i.e.,
\begin{align*}
  \| G_{\tilde L_{k+1,r}}(x^{(k+1,r)}) \|_2^2 \leq \exp\left( -\frac{k}{\sqrt{L_{k,r}/\mu_{k,r}}}\right) \left(
      \frac{4\tilde L_{k+1,r}}{\mu_{k,r}} - \frac{\tilde L_{k+1}}{L_{0,r}}
      + \frac{2\tilde L_{k+1,r}\gamma_{1,r}}{\mu_{k,r}^2} \right) \| G_{L_{0,r}}(x^{(0,r)})
      \|_2^2\leq \bar \epsilon^2 .
\end{align*}
Solving for $k$, we obtain
  \begin{equation}
  \label{eq:no_restart_complexity}
    k = \OC\bigl(\sqrt{Q}\log Q \bigr) + \OC\bigl(
    \sqrt{Q}\log \bar{\epsilon}^{-1} \bigr) ,
  \end{equation}
where we have used $\OC\bigl(\sqrt{L_{k,r}/\mu_{k,r}}\bigr) = \OC\bigl(\sqrt{\tilde{L}_{k+1,r}/\mu_{k,r}}\bigr)
= \OC\bigl(\sqrt{Q}\bigr)$.\\

\subsection{Restart Complexity}
How many iterations are needed before we can detect that a restart is needed?
The restart detection rule \eqref{eq:iteration_mu_sufficiently_small} gives
  \begin{eqnarray*}
    \| G_{\tilde L_{k+1,r}}(x^{(k+1,r)}) \|_2^2 & > &
      \prod_{i=1}^k\left( 1-\sqrt{\frac{\mu_{i,r}}{L_{i,r}}} \right)
      \left( \frac{4\tilde L_{k+1,r}}{\mu_{k,r}} - \frac{2\tilde L_{k+1,r}}{2 L_{0,r}}
      + \frac{2\tilde L_{k+1,r}\gamma_{1,r}}{\mu_{k,r}^2} \right) \| G_{L_{0,r}}(x^{(0,r)})
      \|_2^2 \\
     & \geq &  \left( 1-\sqrt{\frac{\mu_{1,r}}{L_{1,r}}} \right)^k
      \left( \frac{4\tilde L_{1,r}}{\mu_{1,r}} - \frac{2\tilde L_{1,r}}{2 L_{0,r}}
      + \frac{2\tilde L_{1,r}\gamma_{1,r}}{\mu_{1,r}^2} \right) \| G_{L_{0,r}}(x^{(0,r)})
      \|_2^2 \\
     & \geq &  \exp \left ( -\frac{k}{\sqrt{L_{1,r}/\mu_{1,r}}-1} \right)
      \left( \frac{4 L_{1,r}}{\mu_{1,r}} - \frac{2 L_{1,r}}{2 L_{0,r}}
      + \frac{2 L_{1,r}\gamma_{1,r}}{\mu_{1,r}^2} \right) \| G_{L_{0,r}}(x^{(0,r)})
      \|_2^2,
  \end{eqnarray*}
where we have used $L_{i,r}\leq L_{i+1,r}$, $L_{i,r}\leq \tilde L_{i+1,r}$ and
$\mu_{i,r}\geq \mu_{i+1,r}$. Solving for $k$, we obtain
  \begin{equation}
   \label{eq:restart_complexity_k_greater}
    k > \left (\sqrt{\frac{L_{1,r}}{\mu_{1,r}}} - 1 \right)
      \left( \log \left (\frac{4 L_{1,r}}{\mu_{1,r}} - \frac{ L_{1,r}}{L_{0,r}}+\frac{4 \gamma_{1,r} L_{1,r}}{\mu_{1,r}^2} \right)
      + \log \frac{\| G_{L_{0,r}}(x^{(0,r)}) \|_2^2}
        {\| G_{\tilde L_{k+1,r}}(x^{(k+1,r)}) \|_2^2} \right) .
  \end{equation}
Since we do not terminate but restart, we have
$\|G_{\tilde L_{k+1,r}}(x^{(k+1,r)}) \|_2 \geq \bar{\epsilon}$.
After $r$ restarts, in order to satisfy (\ref{eq:restart_complexity_k_greater})
we must have $k$ of the order
  \begin{align*}
    \OC\bigl( \sqrt{Q_r} \bigr) \, \OC(\log Q_r) + \OC \bigl( \sqrt{Q_r} \bigr)
    \, \OC(\log \bar{\epsilon}^{-1} ) ,
  \end{align*}
where
  \begin{equation*}
    Q_r=\OC \left (\frac{L_{1,r}}{\mu_{1,r}}\right )=\OC \left(\rho_\mu^{R-r} Q\right) .
  \end{equation*}
The worst-case number of iterations for running $R$ restarts is then given by
  \begin{eqnarray}
    \nonumber
      \lefteqn{ \sum_{r=0}^{R} \OC\left( \sqrt{Q \rho_\mu^{R-r}} \right) \,
      \OC(\log Q \rho_\mu^{R-r} ) + \OC\left( \sqrt{Q \rho_\mu^{R-r}} \right) \,
      \OC(\log \bar{\epsilon}^{-1} ) } \\
    \nonumber
    & = &\lefteqn{ \sum_{i=0}^{R} \OC\left( \sqrt{Q \rho_\mu^{i}} \right) \,
      \OC(\log Q \rho_\mu^{i} ) + \OC\left( \sqrt{Q \rho_\mu^{i}} \right) \,
      \OC(\log \bar{\epsilon}^{-1} ) } \\
  \nonumber
    & = & \OC\left( \sqrt{Q} \right) \left\{ \sum_{i=0}^{R}
    \OC\left( \sqrt{\rho_\mu^i}\right) \left [ \OC\left( \log Q \rho_\mu^i \right)
    + \OC\left( \log \bar{\epsilon}^{-1} \right) \right ]\right\} \\
  \nonumber
    & = & \OC\left( \sqrt{Q} \right) \left\{ \sum_{i=0}^{R}
    \OC\left( \sqrt{\rho_\mu^i}\right) \Big [ \OC\left( \log Q \right)
    +  \OC\left( \log
      \bar{\epsilon}^{-1} \right) \Big] \right\} \\
  \nonumber
    & = & \OC\left( \sqrt{Q} \right) \left\{ \OC(1) \Big [ \OC\left (\log Q \right)
    +  \OC\left( \log \bar{\epsilon}^{-1} \right) \Big]\right\} \\
  \nonumber
    & = & \OC\left( \sqrt{Q} \right) \OC\left( \log Q \right)
    + \OC\left( \sqrt{Q} \right) \, \OC\left( \log \bar{\epsilon}^{-1} \right) \\
  \label{eq:restart_complexity}
    & = & \OC\left( \sqrt{Q}\log Q \right) + \OC\left( \sqrt{Q}
    \log \bar{\epsilon}^{-1} \right),
  \end{eqnarray}
where we have used
  \begin{align*}
    \sum_{i=0}^{R} \OC \left( \sqrt{\rho_\mu^i} \right) =
    \sum_{i=0}^{R} \OC \left( \sqrt{\rho_\mu}^{\ i} \right) =
      \OC \left( \frac{1-\sqrt{\rho_\mu^{R+1}}}{1-\sqrt{\rho_\mu}}
      \right) = \OC(1) .
  \end{align*}

\subsection{Total Complexity}\label{sec:total_complexity}
The total iteration complexity of \algname{} is given by
\eqref{eq:restart_complexity} plus \eqref{eq:no_restart_complexity}:
  \begin{align}
    \OC\bigl(\sqrt{Q}\log Q \bigr) + \OC\bigl(\sqrt{Q}\log
    \bar{\epsilon}^{-1} \bigr) .
  \end{align}
It is common to write the iteration complexity in terms of reaching an
$\epsilon$-suboptimal solution satisfying $f(x)-f^\star\leq \epsilon$.
This is different from the stopping criteria
$\|G_{\tilde L_{k+1,r}}(x^{(k+1,r)}) \|_2\leq\bar \epsilon$ or
$\|G_{L_{k,r}}(y^{(k,r)}) \|_2\leq\bar \epsilon$ used in
the \algname{} algorithm.
Consequently, we will derive a relation between $\epsilon$ and $\bar \epsilon$.
Using Lemmas \ref{lemma_gradient_map} and 
\ref{lemma_gradient_map_sup3}, in case we stop using
$\|G_{L_{k,r}}(y^{(k,r)}) \|_2\leq\bar \epsilon$ we obtain
  \begin{eqnarray*}
    f\bigl(x^{(k+1,r)}\bigr) - f^\star \leq \left( \frac{2}{\mu}-\frac{1}{2 L_{k,r}}
    \right) \|G_{L_{k,r}}(y^{(k,r)}) \|_2^2 \leq \frac{2}{\mu}
    \|G_{L_{k,r}}(y^{(k,r)}) \|_2^2 \leq \frac{2}{\mu} \bar{\epsilon}^2 ,
  \end{eqnarray*}
and in case we stop using
$\|G_{\tilde L_{k+1,r}}(x^{(k+1,r)}) \|_2\leq\bar \epsilon$, we obtain
  \begin{eqnarray*}
    f\bigl(\tilde x^{(k+1,r)}\bigr) - f^\star \leq \left(
    \frac{2}{\mu} - \frac{1}{2 \tilde L_{k+1,r}}
    \right) \|G_{\tilde L_{k+1,r}}(x^{(k+1,r)}) \|_2^2 \leq \frac{2}{\mu}
    \|G_{\tilde L_{k+1,r}}(x^{(k+1,r)}) \|_2^2 \leq \frac{2}{\mu}
    \bar{\epsilon}^2 .
  \end{eqnarray*}
To return with either $f(\tilde x^{(k+1,r)})-f^\star\leq \epsilon$ or
$f(x^{(k+1,r)})-f^\star\leq \epsilon$ we require the latter bounds to hold and
thus select $(2/\mu) \, \bar{\epsilon}^2=\epsilon$.
The iteration complexity of the algorithm in terms of $\epsilon$ is then
\begin{eqnarray*}
    \OC\Bigl( \sqrt{Q}\log Q \Bigr) + \OC\left(
    \sqrt{Q}\log \left( (\mu\epsilon)^{-1} \right) \right )
    &=&\OC\Bigl( \sqrt{Q}\log Q \Bigr) + \OC\Bigl( \sqrt{Q}\log \mu^{-1}
    \Bigr) + \OC\left(\sqrt{Q}\log  \epsilon^{-1}  \right ) \\
    &=&\OC\Bigl( \sqrt{Q}\log Q \Bigr) + \OC\left(\sqrt{Q}\log \epsilon^{-1}  \right ),
  \end{eqnarray*}
where we have used $\OC\left(  1/\mu\right)=\OC\left(  L/\mu\right)=\OC\left(Q\right)$.

\bibliography{refs}
\bibliographystyle{spmpsci}

\end{document}